\documentclass[11pt,twoside, reqno]{amsart}

\usepackage{charter}

\usepackage{amsmath}
\usepackage{amsthm}
\usepackage{amsfonts, amssymb}
\usepackage{mathrsfs}
\usepackage[all]{xy}
\usepackage{url}

\usepackage{latexsym}
\usepackage[dvips]{graphicx}

\theoremstyle{plain}
\newtheorem{lema}{Lemma}
\newtheorem{prop}[lema]{Proposition}
\newtheorem{teo}[lema]{Theorem}

\theoremstyle{remark}

\newtheorem{obs}[lema]{Remark}
\theoremstyle{definition}

\newtheorem{ej}[lema]{Example}

\newcommand{\Z}{\mathbb{Z}}

\newcommand{\R}{\mathbb{R}}

\newcommand{\F}{F_2}
\newcommand{\Fd}{F_d}

\newcommand{\M}{M_2}

\newcommand{\ZZ}{\Z[X^{\pm 1},Y^{\pm 1}]}
\newcommand{\ZZn}{\Z_n[X^{\pm 1},Y^{\pm 1}]}
\newcommand{\vv}{\textbf{v}}
\newcommand{\hh}{\textbf{h}}

\newcommand{\grid}{\R \times \Z \cup \Z \times \R}
\newcommand{\inv}{\Lambda}
\newcommand{\rk}{\textrm{rk}}

\begin{document}

\title[Invariants for metabelian groups of prime power exponent, colorings and stairs]{Invariants for metabelian groups of prime power exponent, colorings and stairs}

\author[J.A. Barmak]{Jonathan Ariel Barmak $^{\dagger}$}

\thanks{$^{\dagger}$ Researcher of CONICET. Partially supported by grant PICT-2017-2806, PIP 11220170100357CO, UBACyT 20020160100081BA.}

\address{Universidad de Buenos Aires. Facultad de Ciencias Exactas y Naturales. Departamento de Matem\'atica. Buenos Aires, Argentina.}

\address{CONICET-Universidad de Buenos Aires. Instituto de Investigaciones Matem\'aticas Luis A. Santal\'o (IMAS). Buenos Aires, Argentina. }

\email{jbarmak@dm.uba.ar}

\begin{abstract}
We study the free metabelian group $M(2,n)$ of prime power exponent $n$ on two generators by means of invariants $M(2,n)'\to \Z_n$ that we construct from colorings of the squares in the integer grid $\grid$. In particular we improve bounds found by M.F. Newman for the order of $M(2,2^k)$. We study identities in $M(2,n)$, which give information about identities in the Burnside group $B(2,n)$ and the restricted Burnside group $R(2,n)$.
\end{abstract}



\maketitle


\section{Content} \label{content}

In this paper we study the free metabelian group $M(2,n)$ of prime power exponent $n$ in two generators, that is the quotient of the Burnside group $B(2,n)$ by its second derived subgroup. Contrary to Burnside groups, the groups $M(2,n)$ are known to be finite for every $n$. However their order has been determined in very few cases if $n$ is not prime. We define invariants (group homomorphisms) $\Lambda: M(2,n)'\to \Z_n$ and use them to prove that certain identities do not hold in $M(2,n)$ and to improve bounds for its order obtained by M.F. Newman in \cite{New}. Note that the abelianization of $M(2,n)$ has order $n^2$, and since $M(2,n)'$ is finite abelian of exponent $n$, there is a monomorphism $M(2,n)'\to (\Z_n)^r$ (a product of invariants $\Lambda$ as above) for $r$ big enough, and the order of $M(2,n)'$ is then the order of the image of that map. 

Algebraic tools have been used to study Engel identities and the nilpotence class of $M(2,n)$ in works by Bachmuth-Heilbronn-Mochizuki \cite{BHM} and Dark-Newell \cite{DN}. Here we use instead combinatorial ideas based on a construction introduced in \cite{wind}, called the winding invariant. To each word $w$ in the derived subgroup $\F'$ of the free group $\F=F(x,y)$ of rank $2$, we associate a Laurent polynomial $P_w\in \ZZ$. This can be seen as a labeling of finitely many squares in the grid $\grid$ by non-zero integer numbers. We use different colorings of the squares in the grid to define the invariants $M(2,n)'\to \Z_n$.

Some results that we prove are: $[x,y]^{\frac{n}{2}}$ is not a product of $n$th powers in $\F$ if $4|n$ (Proposition \ref{unamas}), $M(2,2^k)$ satisfies the so called $(k+3)$th-Morse identity but not the $(k+1)$th (Proposition \ref{morse}), $n^22^{\frac{5n^2}{16}-1}\le |M(2,n)|\le n^{(n-1)^2}$ if $n \ge 8$ is a power of $2$ (Theorem \ref{cotainf} and Proposition \ref{cotasup}). We also give a lower bound for the order of the restricted Burnside group $R(2,8)$ in Proposition \ref{restric} and obtain a complete invariant for $M(2,4)$ in Examples \ref{ejm24} and \ref{ej210}.

We believe that there is more value in the combinatorial techniques introduced in this article than in the results we have chosen to illustrate how to apply them. This paper is the third in the winding invariant series, after \cite{ac} and \cite{wind} and it is a continuation of the results obtained in \cite{wind} Subsection 6.3. 

\section{Background on the restricted Burnside problem and metabelian Burnside groups} \label{recall}

In 1902 William Burnside asked whether a finitely generated group of exponent $n$ is necessarily finite \cite{Bur}. Any such group $G$ generated by $d$ elements is a quotient of the free group of exponent $n$, $B(d,n)=\langle x_1,x_2, \ldots, x_d| w^n, w\in F(x_1,x_2, \ldots, x_d)\rangle$, also called the Burnside group of exponent $n$ and $d$ generators. Thus Burnside question is about finiteness of Burnside groups. Burnside proved that $B(d,n)$ is finite for $n= 3$ and arbitrary $d$. Sanov proved in \cite{San0} the finiteness of $B(d,4)$ and M. Hall \cite{Hal} did it for $B(d,6)$. In 1968 Adian and Novikov proved that not all Burnside groups are finite. Indeed they proved in \cite{AN} that for $d\ge 2$ and $n$ odd greater than or equal to $4381$, $B(d,n)$ is infinite. Later Adian \cite{Adi} was able to replace the number $4381$ by $665$. Ivanov proved in \cite{Iva} that for $n\ge 2^{48}$ divisible by $2^9$, $B(d,n)$ is infinite. Later Lys\"enok improved this result. In particular there exist only finitely many exponents $n$ for which $B(d,n)$ can be finite. The complete answer to Burnside's problem is still unknown. For instance, it has not been decided yet whether $B(2,5)$ and $B(2,8)$ are finite.

In 1950 Magnus \cite{Mag} formulated a variant of Burnside's problem, known as the restricted Burnside problem: Given $d,n\ge 1$, is there a largest finite quotient $R(d,n)$ of $B(d,n)$? This is equivalent to the following question: are there only finitely many finite groups (up to isomorphism) generated by $d$ elements and with exponent $n$? Indeed if we define $N$ to be the intersection of all the finite-index subgroups of $B(d,n)$, then $N$ is normal and every $d$-generator finite group of exponent $n$ is a homomorphic image of $B(d,n)/N$. The restricted Burnside problem asks whether $B(d,n)/N$ is finite. In that case, $R(d,n)=B(d,n)/N$ is called the \textit{restricted Burnside group} or \textit{universal finite $d$-generator group of exponent $n$}. In 1958 Kostrikin \cite{Kos1,Kos2,Kos3} proved that $R(d,p)$ exists for every prime $p$. This was previous to Adian-Novikov's proof of the existence of infinite Burnside groups and was key in the development of their result (see \cite{Adi2}). P. Hall and Higman \cite{HH} had already reduced the proof of the restricted Burnside problem to the proof of the prime power exponent case under the assumption that there are only finitely many finite simple groups of a given exponent and that the outer automorphism group of any finite simple group is solvable. In 1990 Zel'manov \cite{Zel1,Zel2} gave a positive answer to the restricted Burnside problem for prime power exponent and this, together with Hall-Higman's reduction and the classification of finite simple groups, provided a positive answer to the restricted Burnside problem for arbitrary exponent. One of the key elements in Kostrikin and Zel'manov's work is the relationship with Engel identities in Lie algebras. In fact, for prime exponent $p$, the restricted Burnside problem can be reduced to the proof that every Lie algebra over $\Z_p$ which satisfies the so called Engel identity $E_{p-1}$ is locally nilpotent.

Although $R(d,n)$ is known to exist for each $d,n$, little is known about its structure, in general. In particular the order of $R(d,n)$ is known in very few cases. For each $d$ and each prime power $n$, the $p$-quotient algorithm can be used to obtain a consistent polycyclic/power-commutator presentation of $R(d,n)$ and then to compute its order (see \cite[Section 11.7]{Sim} and \cite{OV, Vau}), however a direct implementation of this algorithm takes in general too much time. Havas, Wall and Wamsley \cite{HWW} proved that $|R(2,5)|=5^{34}$, O'Brien and Vaughan-Lee \cite{OV} proved that $|R(2,7)|=7^{20416}$, Newman and O'Brien \cite{NO} proved that $|B(5,4)|=2^{2728}$, $|R(3,5)|=5^{2282}$. Also $|B(3,4)|=2^{69}$ (\cite{BKW}) and $|B(4,4)|=2^{422}$ (\cite{AHN}). The order of $B(d,4)$ for arbitrary $d$ is still unknown. Upper and lower bounds for $|R(d,n)|$ have been given by Vaughan-Lee and Zel'manov \cite{VZ1,VZ2,VZ3} and by Groves and Vaughan-Lee \cite{GV}.

It has long been known that there are connections between groups of prime power exponent and Engel conditions. The first Engel word is $e_1=[y,x]=yxy^{-1}x^{-1}\in \F$. The $m$th Engel word is defined by $e_m=[y,e_{m-1}]\in \F$. A group $G$ is said to be $m$-Engel if $e_m(a,b)=1\in G$ for every $a,b\in G$ \footnote{When the convention $(x,y)=x^{-1}y^{-1}xy$ is used for commutators, Engel words are defined by $\widetilde{e}_1=(x,y)$ and $\widetilde{e}_{m}=(\widetilde{e}_{m-1},y)$. Both notions are related by the identity $\widetilde{e}_m=e_m(x^{-1},y^{-1})^{-1}$, so the definitions of $m$-Engel group coincide.}. Every group of exponent $3$ is $2$-Engel. This appears in \cite{Lev} but is already implicit in Burnside's paper \cite{Bur}. It is equivalent to saying that $e_2\in \F$ is a product of cubes. Every group of exponent $4$ is $5$-Engel ($e_5$ is a product of fourth powers). Kostrikin observed \cite{Kos60, Kos74} that if the $6$th Engel word $e_6$ is not a product of $5$th powers in $\F$, then $B(2,5)$ is infinite.

%
%

In the quest for finding a solution to the restricted Burnside problem, before Zel'manov's proof, other relations apart from Engel conditions had been searched among groups of prime power exponent. Some papers have explored this path by imposing extra relations to the groups. In particular the metabelian law. Recall that a group $G$ is termed metabelian if $G''=1$. Every $d$-generated metabelian group of exponent $n$ is a quotient of the free metabelian group $M(d,n)$ of exponent $n$. If $G$ is a group and $n\ge 1$, we denote by $G^n$ the subgroup of $G$ generated by the $n$th powers. If $\Fd$ denotes the free group of rank $d$, then $M(d,n)=\frac{\Fd}{\Fd^n\Fd''}$. As explained by Newman in \cite{New}, one advantage of these groups is that they are always finite. Indeed, if we define $R=\Fd^n \Fd'$ and $S=R^nR'$, then $|\Fd:R|=n^d$, so by Schreier formula $\rk(R)=1+(d-1)n^d$. Similarly $|R:S|=n^{\rk(R)}=n^{1+(d-1)n^d}$, so $|\Fd/S|=n^{1+d+(d-1)n^d}$. On the other hand it is easy to see that $S\leqslant \Fd^n \Fd''$. Thus $|M(d,n)|\leq n^{1+d+(d-1)n^d}$. If $n=p^k$ for $p$ a prime, Newman computes also a lower bound for $|M(d,n)|$, which is $p^{1+d(k-1)+(d-1)p^{(k-1)d}}$. This is explained later in this article. Since $M(d,n)$ is finite, its order gives a lower bound for $|R(d,n)|$. Moreover, if we know that a certain identity does not hold in $M(d,n)$ then it will not hold in $R(d,n)$ nor in $B(d,n)$. Recall that for a group $G$, the lower central series is defined by $\gamma_1(G)=G$ and $\gamma_{l+1}(G)=[G, \gamma_l(G)]$ for $l\ge 1$. A group $G$ is nilpotent if there exists $l\ge 0$ such that $\gamma_{l+1}(G)=1$, and the smallest $l$ with this property is the nilpotency class of $G$. If $n$ is a prime power, $M(d,n)$ is nilpotent and its class is known in many cases by results of Meier-Wunderli \cite{Mei}, Bachmuth-Heilbronn-Mochizuki \cite{BHM}, Dark-Newell \cite{DN}, Gupta-Newman-Tobin \cite{GNT}, Newman \cite{New}. In particular the class of $M(2,p^k)$ is $k(p-1)p^{k-1}$. The order of $M(d,n)$ is known when $n$ is prime (see \cite{Mei, New}), but the general case is open. Even for $d=p=2$, Newman's bounds only give $n^22^{\frac{n^2}{4}-1} \le |M(2,n)|\le n^{n^2+3}$ and it is not known, asymptotically, which of the two bounds is closer to reality. Given a positive integer $d$ and a prime power $n$, we can compute $|M(d,n)|$ from any finite presentation using coset enumeration. Moreover, the nilpotent quotient algorithm computes a consistent power-commutator presentation (see \cite{HN,New, Vau}) for $M(d,n)$ from which its order can be directly read off. However, the number of operations needed to complete this task is unknown and $|M(d,n)|$ has been computed in very few cases. Once we have such a presentation it can be decided whether an arbitrary word $w\in F(x_1,x_2, \ldots, x_d)$ is trivial in $M(d,n)$, but again, there is no general method known to do this for arbitrary $w,d,n$ without implementation of this algorithm or other variants. For example, apart from the cases in which $n$ is prime, the order of $M(d,n)$ is known in the following cases: $(d,n)\in\{(2,4),(3,4),(4,4),(5,4),(2,8),(3,8),(4,8),(2,9),(3,9)\}$ (computed by Gupta-Tobin \cite{GT}, Hermanns \cite{Her} and Newman \cite{New}). The order of $M(2,16)$ is known to be smaller than or equal to $2^{376}$ (\cite{New}). For $p$ prime and $n=p^2$, Skopin \cite{Sko85} has shown that $|M(2,n)|\le p^{\frac{3}{2}p^4-\frac{3}{2}p^3-2p^2+p+5}$, which improves Newman's general bound that gives here $|M(2,n)|\le p^{2p^4+6}$. In \cite{Sko90} Skopin proves that the largest class $36$ quotient $M(2,27)/\gamma_{37}(M(2,27))$ of $M(2,27)$ has order less than or equal to $3^{648}$.

\section{Horizontal invariants}


We start by recalling the constructions and results of \cite{wind} that we will need. We will also recall some of the proofs for the sake of completeness and in order to make clear subsequent ideas.

An element $z\in \F'$ in the derived subgroup of the free group $\F$ generated by $x,y$ is a word $x_1^{\epsilon_1}x_2^{\epsilon_2}\ldots x_l^{\epsilon_l}$ in the letters $x,y$ with total exponents $\exp(x,z)=\exp(y,z)=0$. One such word has an associated closed curve $\gamma_z$ in the integer grid $\grid \subseteq \R^2$ which begins in the origin and is a concatenation of curves $\gamma_1, \gamma_2, \ldots , \gamma_l$. The curve $\gamma_i$ moves one unit horizontally if $x_i=x$ and vertically if $x_i=y$, to the right or upwards if $\epsilon_i=1$ and to the left or downwards if $\epsilon_i=-1$. The winding invariant of $z$ is the Laurent polynomial $P_z=\sum\limits_{i,j\in \Z} a_{i,j}X^iY^j \in \ZZ$, where $a_{i,j}$ is the winding number of $\gamma_z$ around $(i+\frac{1}{2}, j+\frac{1}{2})\in \R^2$. For instance, if $z=[x,y]=xyx^{-1}y^{-1}$, the curve $\gamma_z$ traverses the boundary of the square with vertices $(0,0), (1,0), (1,1), (0,1)$ once counterclockwise, so $P_z=1 \in \ZZ$. The winding invariant map $W:\F'\to \ZZ$ which maps $z$ to $P_z$ has many different interpretations \cite{wind}. Clearly $W$ is a group homomorphism. It is surjective and its kernel is $\F''$ (\cite[Theorem 14]{wind}) and thus, it induces an isomorphism $\M' \to \ZZ$ from the derived subgroup of the free metabelian group $\M=\F/\F''$ of rank $2$. This map will also be denoted by $W$. 

An elementary property which follows easily from the definition is:

\begin{obs} \label{remark}
If $z\in \F'$ and $u\in \F$, then $P_{uzu^{-1}}=X^iY^jP_z$ where $(i,j)=(\exp(x,u),\exp(y,u))$. Therefore, $P_{[u,z]}=(X^iY^j-1)P_z$.
\end{obs}

We recall the following result from \cite{wind}:

\begin{prop} \cite[Proposition 8]{wind}  \label{dospotencias}
Let $n\ge 1$. If $z\in \F'$ is a product of two $n$th powers, then $P_z$ is a multiple of $1+m+m^2+\ldots +m^{n-1}$ for some monomial $m=X^iY^j$.
\end{prop}
\begin{proof}
If $z=a^nb^n=(ab)b^{-1}(ab)bb^{-2}(ab)b^2\ldots b^{1-n}(ab)b^{n-1}$, then $ab\in \F'$ and by Remark \ref{remark} $P_{z}=P_{ab}(1+m+m^2+\ldots+m^{n-1})$, for $m=X^{\exp(x,b^{-1})}Y^{\exp(y,b^{-1})}$.
\end{proof}
 
In \cite{wind} we proved that the \textit{area invariant} $A:\F' \to \Z$ defined by $A(z)=P_z(1,1)$, maps a product of $n$th powers in $\F$ to a multiple of $n/2$ when $n$ is even and to a multiple of $n$ when $n$ is odd. In particular a product of $4$th powers has even area. This is not a sufficient condition for an element $z\in \M'$ to be a product of $4$th powers in $\M$. There is another invariant $\kappa:\F' \to \Z$ which, instead of adding all the coefficients of $P_z$, adds some of them and subtracts others. A product of 4th powers must have $\kappa(z)$ divisible by $4$. The main tools that we develop in this article are different generalizations of the invariant $\kappa$, which take into account higher powers of $2$.

In Section \ref{final} we give an example in which the invariant technique successfully describes a variety of groups. We will see that the area invariant completely characterizes the free nilpotent group $N(2,n)$ of class $2$, rank $2$ and exponent $n$ for any $n\ge 1$, determining an isomorphism $\overline{A}:N(2,n)' \to \Z_l$ where $l=\frac{n}{2}$ if $n$ is even and $l=n$ if $n$ is odd. In particular the order $|N(2,n)|$ is $\frac{n^3}{2}$ or $n^3$.  For the free metabelian group $M(2,n)$ of exponent $n$ the problem is much harder and other invariants will be introduced. 

Let $k\ge 1, n=2^k$. A \textit{good coloring} of $\Z_n$ is a function $c:\Z_n\to \{ \textrm{black}, \textrm{white}\}$ which satisfies that for every $a\in \Z_n$, the color $c(a)$ of $a$ is different from the color $c(a+\frac{n}{2})$ of $a+\frac{n}{2}$. A function $\varphi :\Z \times \Z \to \Z_n$ together with a good coloring $c$ induces a coloring of the squares in the grid $\grid$: the square (whose lower left corner is) $(i,j)$ is painted with color $c\varphi (i,j)$. The invariant $\inv=\inv_{\varphi, c}:\ZZ \to \Z_n$ associated with $\varphi$ and $c$ is defined as follows. Given $P \in \ZZ$, define $$\inv(P)=\sum\limits_{(i,j) \textrm{ is black}} P\{(i,j)\}- \sum\limits_{(i,j) \textrm{ is white}} P\{(i,j)\}  \in \Z_n.$$ Here $P\{(i,j)\}$ denotes the coefficient of the monomial $X^iY^j$. The composition $\inv W:\F'\to \Z_n$ will also be denoted by $\inv=\inv_{\varphi, c}$. That is, for $z\in \F'$, $\inv (z)=\inv (P_z)$. Of course, $\inv: \F'\to \Z_n$ is a group homomorphism. Since $\textrm{ker}(W)=\F''$, $\Lambda$ induces a map $\M'=\F' /\F''\to \Z_n$ which is denoted $\inv$ as well. Given a region $A$ of the plane, which is a union of squares in the grid, we define the invariant $\inv(A)$ of $A$ as the number of black squares in $A$ minus the number of white squares, modulo $n$.



\begin{ej} \label{ejemplito}
Let $k=2$, so $n=4$, and let $c:\Z_4 \to \{\textrm{black, white}\}$ be the coloring that paints $0,1$ with black and $2,3$ with white. Let $\varphi :\Z \times \Z \to \Z_4$ be defined by $\varphi (a,b)=b \in \Z_4$. The coloring associated to $\varphi$ and $c$ is illustrated in Figure \ref{ej1}, at the left. All the rows congruent with $0$ or $1$ modulo $4$ are painted with black, while the remaining rows are painted with white.  

\begin{figure}[h] 
\begin{center}
\includegraphics[scale=1.5]{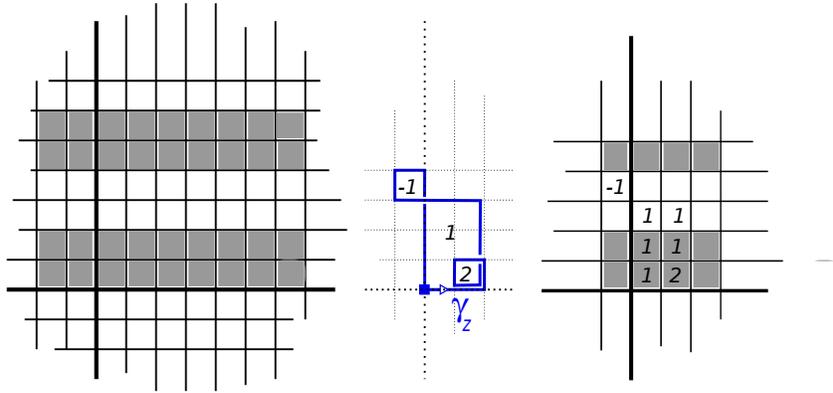}
\caption{The coloring associated to $\varphi$ and $c$, the curve $\gamma_z$ with the winding numbers, and the superposition which computes $\inv (z)$.}\label{ej1}
\end{center}
\end{figure}

Let $z=x^2yx^{-1}y^{-1}xy^3x^{-3}yxy^{-4}\in \F'$. The curve $\gamma_z$ is depicted in the center of Figure \ref{ej1}, together with the (non-zero) winding numbers of the curve around the squares. The winding invariant of $z$ is thus $P_z=1+2X+Y+XY+Y^2+XY^2-X^{-1}Y^3\in \ZZ$. The invariant $\inv_{\varphi,c}(z)$ is just the sum of the coefficients in black squares minus the sum of the coefficients in white squares. That is, $\inv (z)=1+2+1+1-(1+1-1)=0\in \Z_4$.
\end{ej}

\begin{lema} \label{facil}
Let $n,d$ be positive integers such that $2d$ divides $n$. Then the elements of the integer interval $[n]=\{0,1,2,\ldots, n-1\}$ can be matched in such a way that matched elements differ in $d$.   
\end{lema}
\begin{proof}
For $0\le i \le d-1$, match $i$ with $i+d$. These pairs cover the set $[2d]$, and translations of them cover the whole set $[n]$.
\end{proof}

\begin{lema} \label{facil2}
Let $k\ge 1$, $n=2^k$. Let $c:\Z_n \to \{ \textrm{black, white} \}$ be a good coloring. 

(i) Let $a, a+b, a+2b, \ldots, a+(n-1)b$ be an arithmetic progression of length $n$ in $\Z$. When we see this progression in $\Z_n$, the number of black terms minus the number of white terms is a multiple of $n$.

(ii) Let $a, a+b, a+2b, \ldots, a+(\frac{n}{2}-1)b$ be an arithmetic progression of length $\frac{n}{2}$ in $\Z$ with $b$ even. When we see this progression in $\Z_n$ the number of black terms minus the number of white terms is a multiple of $\frac{n}{2}$. Moreover, if $n$ divides $b$, this number is $\frac{n}{2}$ modulo $n$ and if $n$ does not divide $b$, this number is $0$ modulo $n$.
\end{lema}
\begin{proof}
We prove first part (i). If $n$ divides $b$, the $n$ numbers in the progression have the same color. Suppose then that $n \nmid b$. Let $d\in \Z$ be the greatest common divisor of $n$ and $b$. By assumption $d<n$, so $\frac{n}{2d} \in \Z$ and $2 \frac{n}{2d}$ divides $n$. By Lemma \ref{facil}, $[n]$ admits a matching with matched elements of difference $\frac{n}{2d}$. Thus, the set $a+b[n]=\{a,a+b,a+2b, \ldots, a+(n-1)b\}$ admits a matching with matched elements of difference $b\frac{n}{2d}$. But $b\frac{n}{2d}=\frac{b}{d}\frac{n}{2}$ is an odd multiple of $\frac{n}{2}$. Since $c$ is good, matched elements have different colors, so the sequence has the same number of black terms and white terms.

For part (ii), if $n | b$, the $\frac{n}{2}$ terms of the progression have the same color. Otherwise $d=\mathrm{gcd}(n,b)<n$ and $d$ is even by hypothesis. Then $2 \frac{n}{2d}$ divides $\frac{n}{2}$ and by Lemma \ref{facil}, $[\frac{n}{2}]$ can by matched with difference $\frac{n}{2d}$, so $a+b[\frac{n}{2}]$ admits a matching with difference $b\frac{n}{2d}$, and as in the first case matched elements receive opposite colors. 
\end{proof}

Let $n=2^k$ for some $k\ge 2$. Let $c:\Z_n \to \{\textrm{black, white}\}$ be a good coloring. Define a group homomorphism $\varphi:\Z \times \Z \to \Z_n$ by $\varphi (1,0)=0$, $\varphi (0,1)=1$. Let $\hh=\inv_{\varphi, c}:\F' \to \Z_n$ be the associated invariant. It is called the horizontal invariant associated to $c$. In this case all the squares in a row are painted with the same color. If two rows are $\frac{n}{2}$ units away one from the other, then they are painted with opposite colors. This says in particular that among $n$ consecutive squares in a column, half of them are black and the other half are white. If the distance between two rows is a multiple of $n$, they have the same color. The invariant in Example \ref{ejemplito} is of this kind.

\begin{teo} \label{omega}
Let $n=2^k$ for some $k\ge 2$. Let $c$ be a good coloring of $\Z_n$. If $z\in \F'$ lies in $\F^n\F''$, then $\hh (z)=0 \in \Z_n$. In particular, any element in $\F'$ which is a product of $n$th powers of elements in $\F$ has trivial $\hh$-invariant.
\end{teo}
\begin{proof}
The proof is similar to that of Theorem 59 in \cite{wind}. Let $z\in \F' \cap \F^n \F''$. Since $\F'' =\ker W \leqslant \ker \hh$, we may assume $z\in \F' \cap \F^n$. Suppose then $z=u_1^nu_2^n\ldots u_r^n$ for some $u_i\in \F$.

Step 1: We may assume $r=3$. Indeed, we can find $v_2,v_3, \ldots, v_{r-2}\in \F$ such that $u_1^nu_2^nv_2^n$, $v_2^{-n}u_3^nv_3^n$, $v_3^{-n}u_4^nv_4^n, \ldots, v_{r-3}^{-n}u_{r-2}^nv_{r-2}^n$, $v_{r-2}^{-n}u_{r-1}^nu_{r}^n$ are all in $\F'$. Since $\hh:\F'\to \Z_n$ is a group homomorphism, we may then assume that $r=3$. Suppose then that $z=u^nv^nw^n$ for certain $u,v,w\in \F$.

Step 2: $\hh (z)$ depends only on the exponents of $x,y$ in $u,v,w$. If we replace $u$ by another element $\widetilde{u}\in \F$ with same exponents in $x$ and $y$, then $\widetilde{z}=\widetilde{u}^nv^nw^n$ satisfies that $z^{-1}\widetilde{z}$ is a conjugate of an element in $\F'$ which is product of two $n$th powers. Then by Proposition \ref{dospotencias}, $P_{z^{-1}\widetilde{z}}$ is a multiple of $1+m+m^2+\ldots+m^{n-1}$ for some monomial $m$. The same happens if we replace $v$ or $w$ by other elements with the same exponents. So, it suffices to check that for monomials $m=X^iY^j,m'=X^{i'}Y^{j'}$, the polynomial $P=m'(1+m+m^2+\ldots +m^{n-1})$ has invariant $\hh(P)=0$. By definition $\hh(P)$ is (modulo $n$) the number of black elements minus the number of white elements in the sequence $j', j'+j, j'+2j, \ldots , j'+(n-1)j \in \Z_n$. By Lemma \ref{facil2} this is $0$. 

Step 3: The stairs. Let $(a,b)=(\exp (x,u),\exp (y,u))$, $(c,d)=(\exp(x,uv),\exp(y,uv)) \in \Z^2$. Let $R$ be the rectangle circumscribed about the triangle $T$ with vertices $(0,0),(a,b),$ $(c,d)$. By the previous step we can assume $u=x^ay^b$ or $u=y^bx^a$, $v=x^{c-a}y^{d-b}$ or $v=y^{d-b}x^{c-a}$ and $w=x^{-c}y^{-d}$ or $w=y^{-d}x^{-c}$. We choose $u,v,w$ in such a way that the curve $\gamma_{z}$ associated to $z=u^nv^nw^n$ lies outside $T$ (see Figure \ref{rectan}). We must show that the bounded region $T'$ determined by $\gamma_z$ has trivial horizontal invariant, because up to sign this is $\hh(z)$. Since the height of $R$ is a multiple of $n$, $\hh(R)=0\in \Z_n$. Therefore it suffices to prove that the region $R-T'$ has trivial invariant. Now, $R-T'$ consists of three stair-like regions and possibly a rectangular region.

\begin{figure}[h] 
\begin{center}
\includegraphics[scale=0.7]{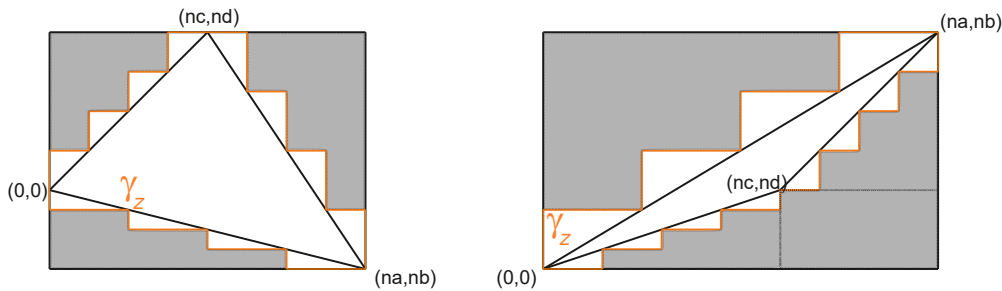}
\caption{The rectangle $R$ circumscribed about $T$. In this picture $n=4$. The curve $\gamma_z$ appears with orange and the region $R-T'$ is shaded.}\label{rectan}
\end{center}
\end{figure}    

Therefore the proof reduces to showing that in any rectangular region with sides divisible by $n$ and in stair-like regions with $n-1$ steps, the number of black squares minus the number of white squares is a multiple of $n$. For a rectangle we have the same argument we used for $R$.

Consider a stair-like region with $n-1$ steps as in Figure \ref{escalerita2}. We must prove that the number of black squares minus the number of white squares is a multiple of $n$, independently of the position of the stair in the integer lattice, the height and the length of the steps.  

\begin{figure}[h] 
\begin{center}
\includegraphics[scale=0.55]{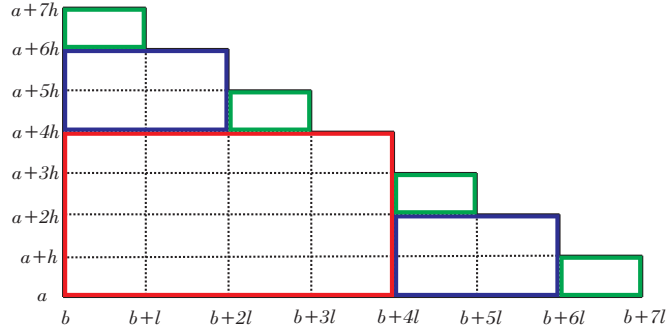}
\caption{A seven-step stair. Here $k=3$. The rectangle of type $0$ appears with red in the picture, those of type $1$ are blue, and those of type $2$ are green.}\label{escalerita2}
\end{center}
\end{figure}

Let $h,l \ge 1$ be the height and the length of the steps in the stair. We divide the stairs in rectangles as shown in Figure \ref{escalerita2}. There are $2^i$ rectangles of size $2^{k-i-1}l \times 2^{k-i-1}h$ for every $0\le i\le k-1$. We call the $2^{k-i-1}l \times 2^{k-i-1}h$ rectangles, rectangles of type $i$. We prove that for each $i$, among all the rectangles of type $i$, the number of black squares minus the number of white squares is a multiple of $2^k$. The first row of one rectangle of type $i$ and the first row of the next rectangle of that type are $2^{k-i}h$ rows apart. Therefore, if $2^i| h$, all the rectangles of type $i$ are painted exactly in the same way. Moreover, in this case each of the columns in these rectangles has height $2^{k-i-1}h$, which is a multiple of $2^{k-1}$, so in particular it is even. It is easy to see then, that in each column the number of black squares minus the number of white squares is even. But there are $2^{k-1}l$ of these columns, so the number of black squares minus the number of white squares among all the rectangles of type $i$ is a multiple of $2^k$. It remains to analyze the case that $2^i$ does not divide $h$. In this case not all the rectangles of type $i$ are painted in the same way. In fact if $2^r$ is the maximum power of $2$ dividing $h$ (so in particular $r\le i-1$) then the $j$th rectangle and the $(j+2^{i-r-1})$th rectangle of type $i$ are $2^{i-r-1}2^{k-i}h=2^{k-r-1}h$ rows apart. But $2^{k-r-1}h$ is a multiple of $2^{k-1}$ and not a multiple of $2^k$. This implies that the two rectangles are painted exactly with opposite colors. Thus the rectangles of type $i$ can be paired in such a way that paired rectangles are painted with opposite colors ($[2^i]$ can be matched with matched elements of difference $2^{i-r-1}$ by Lemma \ref{facil}). Therefore among rectangles of type $i$ there is the same number of black and white squares.    
\end{proof}


In \cite{Str}, Struik proves that for any even positive integer $n$, $[x,y]^{\frac{n}{2}}\in \gamma_3(\F)\F^n$. This is not hard to prove using the area invariant (see Section \ref{final} below). This result is closely related to a result of Sanov \cite{San} and a conjecture by Bruck \cite{Bru}.

Our horizontal invariants can be used to prove that the result is false when we replace $\gamma_3(\F)$ by the smaller subgroup $\F''$.

\begin{prop} \label{unamas}
Let $n$ be a positive integer divisible by $4$. Then $[x,y]^\frac{n}{2} \notin \F'' \F^n$. In particular, $[x,y]^{\frac{n}{2}}$ is not a product of $n$th powers in $\F$.
\end{prop}  
\begin{proof}
Suppose $n=2^kq$ for $q$ odd, $k\ge 2$. Let $c:\Z_{2^k} \to \{\textrm{black, white}\}$ be any good coloring of $\Z_{2^k}$. Let $\hh :\F' \to \Z_{2^k}$ be the horizontal invariant associated with $c$. Then $P_ {[x,y]^{\frac{n}{2}}}=\frac{n}{2} \in \ZZ$, so $\hh([x,y]^{\frac{n}{2}})=\frac{n}{2}=2^{k-1} \in \Z_{2^k}$. Since $\hh ([x,y]^{\frac{n}{2}})$ is non-trivial then by Theorem \ref{omega}, $[x,y]^{\frac{n}{2}}$ is not in $\F''\F^{2^k}$, and then not in $\F''\F^n$.
\end{proof}

\section{Engel and Morse identities}

Given a word $w$ in the free group $F=F(x_1,x_2,\ldots, x_s)$ and a group $G$, we say that $G$ satisfies the identity $w$ if for every homomorphism $\phi:F\to G$ we have $\phi(w)=1$. If $G$ is a quotient of a group $H$ and $G$ does not satisfy identity $w$, then neither does $H$. For instance, from Proposition \ref{unamas} we deduce that the restricted Burnside group $R(2,n)$ does not satisfy the identity $[x,y]^{\frac{n}{2}}$ for any $n$ divisible by $4$. A group is $m$-Engel if it satisfies the identity $e_m$. Note that $B(2,n)$ is $m$-Engel if and only if $e_m\in \F$ is a product of $n$th powers. It is known that $R(2,5)$ is $6$-Engel, so if $e_6$ is not a product of fifth powers, $B(2,5)$ is not finite. Of course, a nilpotent group of class $m$ is $m$-Engel. Conversely, Zorn's Theorem establishes that a finite $m$-Engel group is nilpotent. Not every Burnside group satisfies an Engel identity: $B(2,6)$ is not nilpotent and finite. It is an open problem whether every finitely generated $m$-Engel group is nilpotent (see \cite[14.70]{Kou}). Metabelian groups of prime power exponent are (finite and) nilpotent but if $n$ is not a prime power and $d \ge 2$, $M(d,n)$ is (finite and) not nilpotent (\cite{BHM}). By results of Bachmuth-Heilbronn-Mochizuki \cite{BHM} and of Dark-Newell \cite{DN} it is known that the class of $M(2,p^k)$ is $m=k(p-1)p^{k-1}$ and $M(2,p^k)$ does not satisfy the $(m-1)$-Engel identity. In order to prove that $e_{m-1}\neq 1 \in M(2,p^k)$, Bachmuth, Heilbronn and Mochizuki find a representation of $M(2,p^k)$ by $2\times 2$ matrices using the Magnus embedding. Since Engel identities are completely understood for metabelian groups of prime power exponent, our results are not new and we only exhibit examples that illustrate how to use the invariants introduced in the previous section. In the next section however, we will study basic commutator identities, which generalize Engel identities.





Engel identities (or \textit{congruences} which are a variant for non-nilpotent groups) have been studied for Burnside (non-metabelian) groups. Here the problem becomes harder and there is a still open conjecture by Bruck (see \cite{Bru, GKS, GHMN, GN, San, Kra}).

Since our invariant $\hh$ is trivial in $\F''\F^n$, it can be used to detect non-trivial elements in $M(2,n)$ and to prove that certain identities do not hold in that group. Using Remark \ref{remark} it is easy to prove that the winding invariant of $e_m$ is $P_{e_m}=-(Y-1)^{m-1}$.

\begin{lema} \label{accion}
Let $k\ge 2$. Then $\binom{2^k}{2^{k-1}}-2$ is not a multiple of $8$.
\end{lema}
\begin{proof}
There is an action of the cyclic group $\Z_{2^k}$ on the set $\binom{\Z_{2^k}}{2^{k-1}}$ of subsets of $\Z_{2^k}$ of cardinality $2^{k-1}$, induced by the natural action of $\Z_{2^k}$ on itself. The orbits of this action have order a power of $2$. There is no orbit of cardinality $1$. There is exactly one orbit of cardinality $2$, given by the subset of even numbers in $\Z_{2^k}$ and the subset of odd numbers. There is exactly one orbit of cardinality $4$, it is the orbit of the subset of numbers congruent to $0$ and $1$ modulo $4$ in $\Z_{2^k}$. All the other orbits have cardinality divisible by $8$. Thus, $\binom{2^k}{2^{k-1}}\equiv 6 \mod 8$. 
\end{proof}

\begin{ej} \label{ocho}
Let $n=2^k$ for some $k\ge 3$. Then the $(n+1)$-Engel word $e_{n+1}$ is not a product of $n$th powers. Moreover, $M(2,n)$ is not $(n+1)$-Engel. This follows immediately from Bachmuth, Heilbronn and Mochizuki results: $M(2,n)$ is not $(m-1)$-Engel for $m=k2^{k-1}$, and $m-1\ge n+1$. However we give a short proof of this using our invariant $\hh$. Moreover, the advantage of our approach is that in Proposition \ref{basiccom} we will be able to extend this example to all the basic commutators of weight $\le n+2$, using the notion of \textit{diagonal invariant}.

The winding invariant of $e_{n+1}$ is $-(Y-1)^n$. Take the coloring $\Z_n\to \{\textrm{black,white}\}$ which paints $0,1, \ldots, n/2-1$ with black and $n/2, n/2+1, \ldots, n-1$ with white. We will prove that $\hh(e_{n+1}) \in \Z_n$ is not a multiple of $8$. By definition $\hh(e_{n+1})=\sum\limits_{i=0}^{n/2-1} (-1)^{i+1}\binom{n}{i}-\sum\limits_{i=n/2}^{n-1} (-1)^{i+1}\binom{n}{i}-1$. Since $\binom{n}{i}=\binom{n}{n-i}$ for every $i$, most of the terms in this expression cancel and we obtain $\hh(e_{n+1})=\binom{n}{n/2}-2$. The result follows then from Lemma \ref{accion} and Theorem \ref{omega}.
\end{ej}

\begin{ej} \label{ej25}
$M(2,16)$ is not a $25$-Engel group. In fact Bachmuth-Heilbronn-Mochizuki's result says that it is not $31$-Engel. Our proof is easy by considering the invariant $\hh$ as above. We only have to check that $\hh(e_{25})\neq 0$. The winding invariant of $e_{25}$ is $-(Y-1)^{24}=\sum\limits_{i=0}^{24} (-1)^{i+1}\binom{24}{i}Y^i$. Modulo $16$ the coefficients of this polynomial are: $-1,8,-4,8,$ $-2,8,$ $-4,8,$ $1,0,$ $8,0,$ $4,0,8,0,1,8,-4,8,-2,8,-4,8,-1$. Thus, $\hh(e_{25})=(-1+8-4+8-2+8-4+8)-(1+0+8+0+4+0+8+0)+(1+8-4+8-2+8-4+8)-(-1)=5-5+7+1=8\in \Z_{16}$.
\end{ej}


%

We define the following $n/2$ colorings of $\Z_n$, $c_0,c_1,\ldots, c_{n/2-1}:\Z_n \to \{\textrm{black, white}\}$. The coloring $c_i$ paints $i,i+1,\ldots, i+n/2-1$ with black and the remaining elements with white. The associated horizontal invariants are $\hh^0$, $\hh^1, \ldots, \hh^{n/2-1}$. Similarly we have the vertical invariants $\vv^0$, $\vv^1,\ldots , \vv^{n/2-1}$. The invariant $\vv^i$ is associated to the coloring of the squares in the grid which paints the columns congruent to $i,i+1,\ldots, i+n/2-1$ with black and the remaining columns with white. That is $\vv^i(z)=\hh^i(P_z(Y,X))$. By Theorem \ref{omega}, the map $\Omega: \F' \to (\Z_n)^n=\Z_n \times \ldots \times \Z_n$ defined by $$\Omega(z)=(\hh^0(z), \hh^1(z), \ldots, \hh^{n/2-1}(z), \vv^0(z), \vv^1(z), \ldots, \vv^{n/2-1}(z))$$ is trivial in $\F''\F^n$ so it induces a well-defined homomorphism $\Omega : M(2,n) \to (\Z_n)^n$.

\begin{lema} \label{gama}
$\Omega$ is trivial in $\gamma_{j+1}(\F)$ if and only if $\Omega(e_j)=0$.
\end{lema}
\begin{proof}
Since $e_j\in \gamma_{j+1}(\F)$, one implication is trivial. Suppose now that $\Omega(e_j)=0$. Then by symmetry $\Omega (e_j(y,x))=0$. Let $z\in \gamma_{j+1}(\F)$. Using Remark \ref{remark} it is easy to see that $P_z$ is a sum of multiples of polynomials of the form $(m_1-1)(m_2-1)\ldots (m_{j-1}-1)$ where the $m_i$ are monomials (see also \cite[Proposition 23]{wind}). If $P \in \ZZ$ and $m$ is a monomial, $\Omega (P)$ and $\Omega (m P)$ have the same coordinates up to a shift and change of sign.  It suffices to prove then that for a polynomial $P=(m_1-1)(m_2-1)\ldots (m_{j-1}-1)$, where the $m_i$ are monomials, $\Omega$ vanishes. Now, we can consider only the invariants $\hh^0, \hh^1, \ldots, \hh^{n/2-1}$ since for the $\vv^i$ the proof is symmetrical. Since the colors of the squares in the integer lattice used to compute $\hh^i$ depend only on the row, we may also assume that each monomial $m_i$ is a power of $Y$. Also, since $Y-1$ divides $Y^l-1$ for every $l\in \Z$, $P$ is a multiple of $(Y-1)^{j-1}=-P_{e_j}$, so as $\Omega$ vanishes at $e_j$, it also vanishes at $z$.
\end{proof}

%

We consider the map $\overline{\Omega}:\F'\to \Z_{n/2} \times (\Z_n)^n$ defined by $\overline{\Omega}(z)=(A(z), \Omega(z))$. Recall that $A$ denotes the area invariant $A:\F' \to \Z_{n/2}$ defined by $A(z)=P_z(1,1)$, which is trivial in $\F''\F^n$ by Theorem 59 of \cite{wind}. 

\begin{prop} \label{orderimage}
The image $\textrm{Im}(\Omega)$ of $\Omega :\F' \to (\Z_n)^n$ is a subgroup of order $2(n/2)^n$ and $\textrm{Im}(\overline{\Omega})$ has order $(n/2)^{n+1}$.
\end{prop}
\begin{proof}
Let $z\in \F'$ be an element with winding invariant $P_z=1-Y^{-1}$. It is easy to see that $\Omega (z)=(2,0,0, \ldots, 0) \in \textrm{Im}(\Omega)$. For $1\le i\le n/2-1$, the elements $y^{i}zy^{-i}$ give permutations of the coordinates of this vector with the non-trivial coordinate in any of the first $n/2$ positions. Changing $Y$ by $X$ we obtain the remaining permutations of the coordinates. This already proves that $\textrm{Im}(\Omega)$ contains a subgroup $D$ of order $(n/2)^n$. The element $z=1$ has $\Omega(z) \notin D$. On the other hand it is easy to see that for any $z\in \F'$, all the coordinates of $\Omega (z)$ have the same parity, so $|D|<|\textrm{Im}(\Omega)|\le 2(n/2)^n$ and the result follows. 

The area of an element $z\in \F'$ with winding invariant $P_z=1+X^{n/2}Y^{n/2}$ is $2$, while $\Omega(z)=0$. Thus, $|\textrm{Im}(\overline{\Omega})|\ge 2(n/2)^{n}n/4=(n/2)^{n+1}$. The parity of the area is also determined by the value of $\Omega$, so the other inequality holds as well. 
\end{proof}

Recall that $M(2,16)$ is nilpotent of class $m=32$. In Example \ref{ej25}, we proved that $\hh^0(e_{25})\neq 0 \in \Z_{16}$. It is easy to prove that $\hh^0(e_{26})=0$ and moreover that in fact $\Omega(e_{26})=0$. By Lemma \ref{gama}, $\Omega$ is trivial in $\gamma_{27}(\F)$. The area invariant is already trivial in $\gamma_3(\F)$. Thus $\overline{\Omega}$ induces a map $\overline{\Omega}: M(2,16)'/\gamma_{27}(M(2,16))\to \Z_{8} \times (\Z_{16})^{16}$ and from Proposition \ref{orderimage} we deduce that $|\frac{M(2,16)'}{\gamma_{27}(M(2,16))}|\ge 2^{51}$. On the other hand $|\frac{M(2,16)}{M(2,16)'}|=16^2=2^8$. Therefore we deduce the following 

\begin{prop} \label{p16}
The order $|M(2,16)/\gamma_{27}(M(2,16))|$ of the largest class $26$ quotient of $M(2,16)$ is greater than or equal to $2^{59}$.
\end{prop}

A similar result can be obtained for $n=8$. Here $\Omega(e_{10})=0\in \Z_8$, so $\overline{\Omega}$ is trivial in $\gamma_{11}(\F)$ and we deduce from Proposition \ref{orderimage}:

\begin{obs}
The order of $M(2,8)/\gamma_{11}(M(2,8))$ is greater than or equal to $2^{24}$.
\end{obs}

In fact something stronger is already known. $M(2,8)$ has nilpotency class $12$ and its order has been computed by Hermanns \cite{Her}: $|M(2,8)|=2^{63}$. Skopin computed in \cite{Sko85} the orders of all the factors $\gamma_l(M(2,8))/\gamma_{l+1}(M(2,8))$, from which it can be seen that $M(2,8)/ \gamma_{11}(M(2,8))$ has order $2^{54}$. As far as we know the order of $M(2,16)$ has not been determined, but $2^{87} \le |M(2,16)|\le 2^{376}$ (see \cite{New} and next section). A GAP \cite{GAP} program using the package NQ \cite{NQ} returns after 120 hours that $|M(2,16)/\gamma_{17} (M(2,16))|=2^{185}$, which is of course much better than our result in Proposition \ref{p16}, and improves the lower bound given above. The computation that $|M(2,16)/\gamma_{9}(M(2,16))|=2^{70}$ takes seconds. 

\medskip

For $n=4$, Proposition \ref{orderimage} implies that $|M(2,4)|\ge 2^9$. In fact $|M(2,4)|=2^{10}$, so $\Omega$ (and $\overline{\Omega}$) is very close to a complete invariant that could be used to solve the word problem in $M(2,4)$ in an explicit way. In the next section we will introduce invariants more general that the horizontal and vertical invariants, though based on the same coloring ideas, and we will use them to obtain a complete invariant for $M(2,4)$.

\medskip
  
A \textit{semigroup identity} is an identity of the form $uv^{-1}$, where $u,v\in F(x_1,x_2,\ldots, x_s)$ are positive words (no letter appears with negative exponent). The Burnside identity $x^n$ is an example of this type. Another is the ($m$th) \textit{Morse identity} $u_mv_m^{-1}\in \F$, where $u_1=x$, $v_1=y$ and $u_{i+1}=u_iv_i$, $v_{i+1}=v_iu_i$ for every $i\ge 1$. In some sense Burnside identities and Morse identities generate all the semigroup identities (see \cite[Theorem 5.7]{Sha}). Finite groups satisfying a Morse identity are completely characterized by a theorem of Boffa and Point \cite[Corollaire]{BP}: They are the extensions of a nilpotent group by a 2-group. The Morse identities are related to the original proof of Novikov and Adian that there are infinite Burnside groups (see \cite[Section 6]{All}).

It is easy to see that the total exponents of $x$ and of $y$ in both $u_m$ and $v_m$ is $2^{m-2}$ if $m\ge 2$. Let $m\ge 3$. Since $u_{m}v_{m}^{-1}=u_{m-1}v_{m-1}u_{m-1}^{-1}v_{m-1}^{-1}$ and $v_{m-1}u_{m-1}^{-1}\in \F'$, then the winding invariant $P_{u_{m}v_{m}^{-1}}=P_{u_{m-1}v_{m-1}^{-1}}+X^{2^{m-3}}Y^{2^{m-3}}P_{v_{m-1}u_{m-1}^{-1}}=(1-X^{2^{m-3}}Y^{2^{m-3}})P_{u_{m-1}v_{m-1}^{-1}}$. By induction we obtain then $$P_{u_mv_m^{-1}}=(1-X^{2^{m-3}}Y^{2^{m-3}})(1-X^{2^{m-4}}Y^{2^{m-4}})\ldots (1-XY).$$

\begin{prop} \label{morse}
Let $k\ge 2$, $n=2^k$. Then $M(2,n)$ satisfies the $(k+3)$th Morse identity but it does not satisfy the $(k+1)$th Morse identity. In particular, the restricted Burnside group $R(2,n)$ does not satisfy the $(k+1)$th Morse identity, and neither does $B(2,n)$.
\end{prop}
\begin{proof}
In order to check that $M(2,n)$ satisfies the $(k+3)$th Morse identity we only have to show that $u_{k+3}v_{k+3}^{-1}=1\in M(2,n)$. The winding invariant of $u_{k+3}v_{k+3}^{-1}$ is $P_{u_{k+3}v_{k+3}^{-1}}=(1-X^nY^n)Q$ for some polynomial $Q$. Let $w\in \F'$ be such that $P_w=Q$. Then $P_{[w,(xy)^n]}=(1-X^nY^n)Q$. Thus, $u_{k+3}v_{k+3}^{-1}=[w,(xy)^{n}] \mod \ker(W)=\F''$, and since $[w,(xy)^n]$ is a product of $n$th powers, $u_{k+3}v_{k+3}^{-1}\in \F''\F^n$, as we wanted to show.

To prove that $M(2,n)$ does not satisfy the $(k+1)$th Morse identity it suffices to prove that $\Omega (u_{k+1}v_{k+1}^{-1})\neq 1$, where $\Omega: \F' \to (\Z_{n})^{n}$ is the invariant defined above. Now, $P_{u_{k+1}v_{k+1}^{-1}}=1+R$, where $R$ is a polynomial with monomials $X^lY^l$ for $1\le l\le 2^{k-1}-1=n/2-1$. Thus the difference between $\hh^0(u_{k+1}v_{k+1}^{-1})$ and $\hh^1(u_{k+1}v_{k+1}^{-1})$ is $2\in \Z_n$. In particular one of them has to be non-zero, so $u_{k+1}v_{k+1}^{-1}$ is non-trivial in $M(2,n)$.  
\end{proof}


\section{More invariants}

In the previous section we have deduced some lower bounds for the order of $M(2,n)$ via the invariant $\overline{\Omega}$ which is composed by $n$ invariants $M(2,n)'\to \Z_n$ and the area invariant $M(2,n)'\to \Z_{\frac{n}{2}}$. We proved (a refinement of) the inequality $|M(2,n)|\ge n^2 (\frac{n}{2})^{n+1}$. On the other hand, Newman's remarks in \cite{New} show that $n^22^{\frac{n^2}{4}-1}\le|M(2,n)|\le n^{n^2+3}$. This suggests that we should be able to find quadratically many invariants $M(2,n)'\to \Z_n$ and not just linearly many. As explained in Section \ref{content}, there exists a monomorphism $M(2,n)'\to (\Z_n)^r$ for $r$ big enough, so it makes sense to look for more invariants of this type. The aim of this section is to find and make explicit some of those maps. 

\begin{teo} \label{muchosinvariantes}
Let $k\ge 2$, $n=2^k$. Let $c$ be a good coloring of $\Z_n$ and let $\varphi:\Z \times \Z\to \Z_n$ be a homomorphism such that $\varphi(1,0)$, $\varphi(0,1) \in \Z_n$ are odd. Let $t:\Z \times \Z \to \Z \times \Z$ be a translation by any vector $(i_0,j_0)\in \Z \times \Z$. Let $\inv=\inv_{\varphi t, c}$ be the associated invariant. Then for any $z\in \F'$ which is a product of $n$th powers in $\F$, $\inv(z)=0\in \Z_n$.
\end{teo}
\begin{proof}
Note that $\inv_{\varphi t,c}(z)=\inv_{\varphi t,c}(P_z)=\inv_{\varphi ,c}(X^{i_0}Y^{j_0}P_z)=\inv_{\varphi, c}(x^{i_0}y^{j_0}zy^{-j_0}x^{-i_0})$. Thus, we may assume $i_0=j_0=0$.

Just as in Theorem \ref{omega} we can assume $z$ is a product of three $n$th powers. Then we must show that for any polynomial of the form $P=m'(1+m+m^2+\ldots+m^{n-1})$, the invariant $\inv$ is trivial. But if $m'=X^{i'}Y^{j'}$ and $m=X^iY^j$, then $\inv(P)$ is the number of black terms minus the number of white terms in the progression $\varphi(i',j'), \varphi(i',j')+\varphi(i,j), \varphi(i',j')+2\varphi(i,j), \ldots , \varphi(i',j')+(n-1)\varphi(i,j)$, which is $0$ modulo $n$ by Lemma \ref{facil2}.

We move to step 3 in the proof of Theorem \ref{omega}. Note that in any $n$ consecutive squares in a row or a column, there are exactly $\frac{n}{2}$ which are black and $\frac{n}{2}$ which are white. This is because $\varphi(1,0)$ and $\varphi(0,1)$ are odd. Thus, for any rectangle with side lengths multiple of $n$ the number of black squares is the same as white squares. So we consider once again the stair-like regions.

Because this coloring is not symmetric with respect to horizontal and vertical reflections, we need to consider in principle four types of stairs: one that we walk up from right to left and the three rotations of these, by $\pi/2, \pi$ and $3\pi/2$. Again we must show that the number of black squares minus the number of white squares is a multiple of $n$. The proof will be by induction in the length $l$ and the height $h$ of the steps.

Case 1: $h$ is odd and $l$ is even. Assume first that the stair is in the position of Figure \ref{imparpar}, so one climbs up from right to left. We divide the stair in four regions, $A,-A,B,C$. Region $A$ consists of the top $n/2-1$ levels of the stair, while $-A$ is region $A$ translated downwards $\frac{n}{2}h$ units. Since $h$ is odd, $\varphi(0,1)$ is odd and $c$ is good, then $A$ and $-A$ are painted with opposite colors. Region $B$ consists of the squares below $A$ which are not in $-A$. Region $B'$ is outside the stairs and is the translation of $B$ by $\frac{n}{2}l$ units to the right. Since $l$ is even, $B$ and $B'$ are painted identically, so the invariant $\inv$ of the stair is the invariant of the rectangle $C\cup B'$, where $C$ are the squares of the stair which are not in $A,-A,B$. But this rectangle has one side of length $\frac{n}{2}l$, which is a multiple of $n$, so its invariant $\inv(C\cup B')$ is trivial. 

\begin{figure}[h] 
\begin{center}
\includegraphics[scale=0.7]{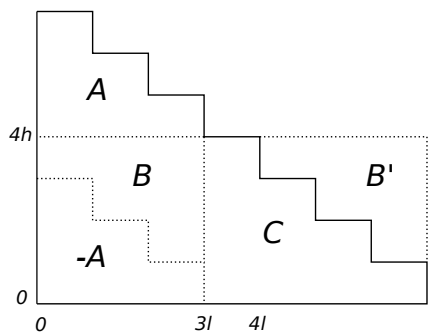}
\caption{An $(n-1)$-step stair for $n=8$, $h$ odd, $l$ even.}\label{imparpar}
\end{center}
\end{figure}

If the stair is in a different position (climbing from left to right, or the corresponding upside-down stairs), the same proof works. We have not used an induction argument for this case.

Case 2: $h$ even, $l$ odd, is symmetric to case 1.

Case 3: $h$ and $l$ even. In this case we consider four $(n-1)$-step stairs contained in our big stair (see Figure \ref{parpar}). The height of the steps in these four small stairs is $h/2$ and the length is $l/2$. Region $A$ is disjoint form $B,C,D$, and $D$ is also disjoint from $B$ and $C$. But $B$ and $C$ overlap in $n-1$ steps of size $l/2\times h/2$. They are marked with the sign $-$ in the figure. There are $n-1$ rectangles of size $l/2\times h/2$ in the big stair which are not in $A\cup B\cup C\cup D$. They are marked with sign $+$. 

\begin{figure}[h] 
\begin{center}
\includegraphics[scale=0.7]{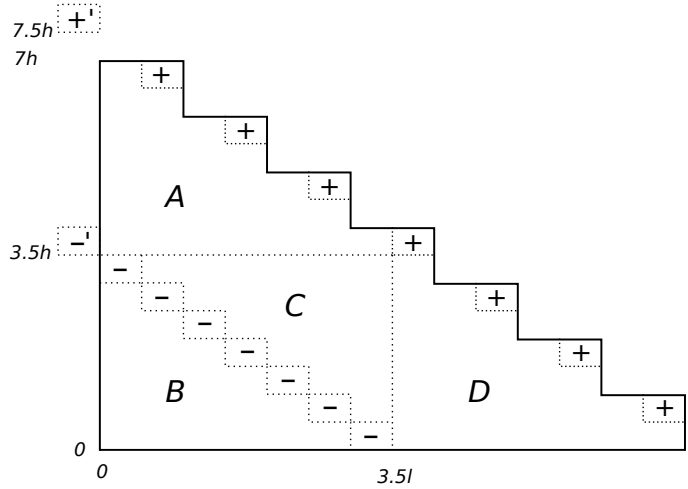}
\caption{An $(n-1)$-step stair for $n=8$, $h$ and $l$ even.}\label{parpar}
\end{center}
\end{figure}

By induction the invariants $\inv(A), \inv(B), \inv(C)$ and $\inv(D)$ are zero, so the invariant in the big stair is the invariant in the union of the $+$ regions minus the invariant in the $-$ regions. We add two $l/2\times h/2$ rectangles to the picture, those marked with $+'$ and $-'$. Since they are $\frac{n}{2}h$ squares apart and $h$ is even, they are painted with the same color. Now, the $+$ rectangles together with the $+'$ rectangle are $n$ rectangles which are aligned and the distance between two consecutive is constant. Thus the characteristic polynomial of this region (the sum of the monomials $X^iY^j$ for squares $(i,j)$ of the region) is a multiple of a polynomial of the form $1+m+m^2+\ldots+m^{n-1}$, so we already know that the invariant is trivial. The same happens with the region formed by the $-$ rectangles together with the $-'$ rectangle. Thus, the invariant of the big stair is also trivial.

Case 4: $h$ and $l$ odd. Just as in case 1, we consider the regions $A$ and $-A$ inside the stair (see Figure \ref{imparimpar}). Since $h$ is odd, they are painted with opposite colors, and we only care about the region $B$ which consists of the squares which are not in $A$ nor in $-A$.

\begin{figure}[h] 
\begin{center}
\includegraphics[scale=0.7]{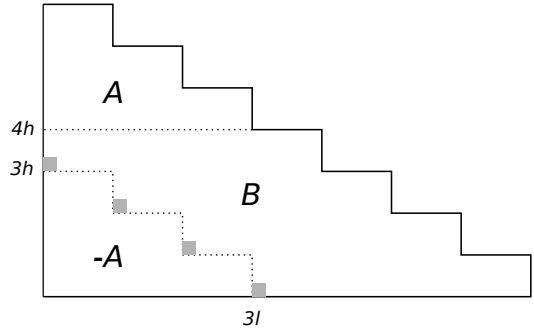}
\caption{An $(n-1)$-step stair for $n=8$, $h$ and $l$ odd.}\label{imparimpar}
\end{center}
\end{figure}

Consider the region $C$ formed by the squares of $B$ which satisfy that the neighboring squares to the left and below are not in $B$ (they appear with grey in Figure \ref{imparimpar}). The characteristic polynomial of $C$ is $P=m'(1+m+m^2+\ldots+m^{\frac{n}{2}-1})$ for $m=X^lY^{-h}$ and some monomial $m'$. In order to understand the invariant $\inv(P)$ we should analyze the colors of the numbers $a,a+b,a+2b,\ldots, a+(\frac{n}{2}-1)b \in \Z_n$, where $a$ depends on $m'$ and $b=\varphi(l,-h)$. By assumption, $b$ is even, so Lemma \ref{facil2} says that $\inv(C)$ is a multiple of $\frac{n}{2}$. Moreover, $\inv(C)=\frac{n}{2}\in \Z_n$ if $n|b$ and $\inv(C)=0$ if $n\nmid b$. Region $B$ can be covered by $hl\frac{n}{2}$ translates of region $C$. Each of them has the same invariant as $C$. Since $k\ge 2$, $hl\frac{n}{2}$ is even and $\inv(B)=hl\frac{n}{2}\inv(C)=0\in \Z_n$.
\end{proof}

If $\varphi :\Z \times \Z \to \Z_n$, $c$ and $(i_0,j_0)$ are in the hypotheses of Theorem \ref{muchosinvariantes}, then $\inv_{\varphi t,c}:\F' \to \Z_n$ induces a well defined group homomorphism $\inv_{\varphi t,c}: M(2,n)' \to \Z_n$.

\begin{ej} \label{ejm24}
In this example we exhibit a complete invariant for $M(2,4)$, that is a concrete solution to the word problem. We describe a method which decides for any $u\in \F$, whether $u$ is trivial in $M(2,4)$. Of course, this can be obtained from any finite presentation and coset enumeration or a consistent power-commutator presentation, but our methods are elementary and very easy to carry out by hand in short time. 

Remember that we constructed the map $\Omega:M(2,4)'\to (\Z_4)^4$ whose image has order $2^5$. Consider the map $\widetilde{\Omega}:M(2,4)' \to (\Z_4)^5$ defined by $\widetilde{\Omega}(w)=(\Omega (w), \Lambda (w))$, where $\Lambda:M(2,4)'\to \Z_4$ is the \textit{diagonal invariant} defined as in the statement of Theorem \ref{muchosinvariantes} for $\varphi(1,0)=\varphi(0,1)=1$, $(i_0,j_0)=(0,0)$, and $c:\Z_4\to \{\textrm{black, white}\}$ the coloring that paints $0,1$ with black and $2,3$ with white. Figure \ref{complet} shows the five invariants that compose $\widetilde{\Omega}$. If $w\in \F'$ is such that $P_w=1+X^2Y^2$, then $\widetilde{\Omega}(w)=(0,0,0,0,2)$. This already implies that the order $|\textrm{Im} (\widetilde{\Omega})|\ge 2 |\textrm{Im} (\Omega)|=2^6$. On the other hand, since $|M(2,4)|=2^{10}$, then $|M(2,4)'|=2^6$, so $\widetilde{\Omega}$ is a monomorphism.

\begin{figure}[h] 
\begin{center}
\includegraphics[scale=0.9]{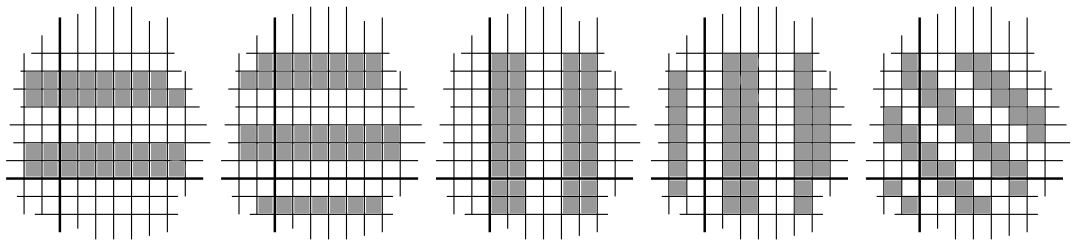}
\caption{The five colorings that form a complete invariant for $M(2,4)$: two horizontal invariants, two vertical invariants and a diagonal invariant.} \label{complet}
\end{center}
\end{figure}

The algorithm to decide whether a given element $u\in \F$ is trivial in $M(2,4)$ is the following: compute the total exponents $a=\exp (x,u)$, $b=\exp (y,u)$. If any of them is not divisible by $4$, $u$ is non-trivial in $M(2,4)$. Otherwise, let $v=x^{-a}y^{-b}u\in \F'$. Compute $\widetilde{\Omega}(v)$. If this is non-trivial, $v$ and then $u$ are non-trivial in $M(2,4)$. If it is trivial, then $u=v=1\in M(2,4)$.
\end{ej}

In the previous example we have used that $|M(2,4)|=2^{10}$. In fact the inequality $|M(2,4)|\ge 2^{10}$ follows from the fact that $|\textrm{Im}(\widetilde{\Omega})|\ge 2^6$. In Example \ref{ej210} we will prove the inequality $|M(2,4)|\le 2^{10}$ by different methods, thus providing an alternative proof that $|M(2,4)|=2^{10}$.

Let $n=2^k$. As we recalled in Section \ref{recall} and the beginning of this section, Newman proves in \cite{New} that $|M(2,n)|\ge n^22^{\frac{n^2}{4}-1}$. His short argument is as follows. Take $R=\F^{\frac{n}{2}}\F'\leqslant \F$ and $S=R^2R'\leqslant \F$. Then $\F'\leqslant R$, so $\F''\leqslant R'\leqslant S$ and $\F^n\leqslant R^2 \leqslant S$. Thus, $\F^n \F'' \leqslant S$, which shows that the order of $M(2,n)=\F/\F^n\F''$ is greater than or equal to the order of $\F/S$. But the latter is easy to compute. Note that $\F/R$ is an abelian group of order $(\frac{n}{2})^2$, so by Schreier formula $R$ is a free group of rank $1+(\frac{n}{2})^2$. Once again, $R/S$ is abelian of order $2^{\textrm{rk}(R)}=2^{1+(\frac{n}{2})^2}$ and the order of $\F/S$ is $|\F/R||R/S|=(\frac{n}{2})^22^{\frac{n^2}{4}+1}$.

Our next result improves Newman's lower bound just a little bit.

\begin{teo} \label{cotainf}
Let $k\ge 3$, $n=2^k$. Then $|M(2,n)|\ge n^2 2^{\frac{5n^2}{16}-1}$.
\end{teo}
\begin{proof}
Let $S=R^2R'$ be as above. It suffices to prove then that $|S/\F^n\F''|\ge 2^{\frac{n^2}{16}}$. We will show that $|\frac{S\cap \F'}{\F^n\F'' \cap \F'}|\ge 2^{\frac{n^2}{16}}$. Given any Laurent polynomial $P\in \ZZ$, there exists $z\in \F'$ with $P_z=P$. Thus, $z^2\in S\cap \F'$ and $2P\in W(S\cap \F') \leqslant \ZZ$.

Let $O=\{1,9,17,\ldots, n-7\}$ be the set of integers congruent to $1$ modulo $8$ between $0$ and $n-1$. Let $c_0,c_1:\Z_n\to \{\textrm{black, white}\}$ be two good colorings which differ only in $0$ and $n/2$. 

Consider the action of $\Gamma=\Z_8 \times \Z_2$ on $\Z_n \times \Z_n$ where the generator of $\Z_8$ maps a pair $(i,j)\in \Z_n\times \Z_n$ to $(i-n/8, j+n/8)$ and the generator of $\Z_2$ maps $(i,j)$ to $(i+n/2,j)$. Let $\Xi$ be a set of representatives of the orbits of this action. This set has $n^2/16$ elements.

Let $\phi: M(2,n)' \to \Z_n^{O\times \{c_0,c_1\} \times \Xi}$ be the group homomorphism defined by $\phi (z)_{b,c,(i_0,j_0)}=\inv_{\varphi_{1,b}t_{-(i_0,j_0)},c}(z)$, where $\varphi_{1,b}:\Z \times \Z \to \Z_n$ is the homomorphism determined by $\varphi_{1,b}(i,j)=i+bj$ and $t_{-(i_0,j_0)}$ is the translation by the vector $-(i_0,j_0)$ (any vector in the fiber of $\Z \times \Z \to \Z_n \times \Z_n$ over $-(i_0,j_0)$). By Theorem \ref{muchosinvariantes}, $\phi$ is a well-defined group homomorphism. To prove the result it suffices to show that the order of $\phi (S\cap \F'/\F^n\F''\cap \F')$ is greater than or equal to $2^{\frac{n^2}{16}}$. Let $\overline{\phi}: \ZZ \to \Z_n^{O\times \{c_0,c_1\} \times \Xi}$ be defined by $\overline{\phi}(P)_{b,c,(i_0,j_0)}=\inv_{\varphi_{1,b}t_{-(i_0,j_0)},c}(P)$, so we have $\overline{\phi}W=\phi q$, where $q:\F'\to M(2,n)'$ is the projection. Then $\phi(S\cap \F'/\F^n\F'' \cap \F')=\overline{\phi}W(S\cap \F')\geqslant \overline{\phi}(2\ZZ)$. The map $\overline{\phi}$ factorizes as $\ZZ \twoheadrightarrow \ZZn \to \Z_n^{O\times \{c_0,c_1\} \times \Xi}$ and we are going to study the restriction $\psi$ of the map $\ZZn\to \Z_n^{O\times \{c_0,c_1\} \times \Xi}$ to the subgroup $H$ of Laurent polynomials with coefficients in $\Z_n$ supported in the square $[n]^2$ and with all coefficients even, i.e. $P\{(i,j)\}$ is even for every $i,j\in \Z$ and $P\{(i,j)\}\neq 0$ implies that $0\le i,j \le n-1$. Note that the order of $H$ is $\left(\frac{n}{2} \right)^{n^2}$. We want to prove that $|\textrm{Im}(\psi)|\ge 2^{\frac{n^2}{16}}$.

We have a sequence 

$$ \textrm{ker}(\psi) \hookrightarrow H \overset{\psi}{\to} \Z_n^{O\times \{c_0,c_1\} \times \Xi}.$$ 

\noindent Let $P \in \textrm{ker}(\psi)$. This means that for every $b\in O$, $c=c_0,c_1$ and $(i_0,j_0)\in \Xi$, 

$$ \sum\limits_{c(i-i_0+b(j-j_0)) \textrm{ is black}} P\{(i,j)\} - \sum\limits_{c(i-i_0+b(j-j_0)) \textrm{ is white}} P\{(i,j)\}=0\in \Z_n.$$ 

\noindent If we subtract the equations corresponding to $c_0$ and to $c_1$ we obtain 

$$ 2\sum\limits_{i-i_0+b(j-j_0)=0} P\{(i,j)\} - 2\sum\limits_{i-i_0+b(j-j_0)=n/2} P\{(i,j)\}=0\in \Z_n.$$

\noindent Here the equalities $i-i_0+b(j-j_0)=0$ and $i-i_0+b(j-j_0)=n/2$ are considered modulo $n$.

We make $b$ vary among all the elements in $O$ and add all these equations to obtain

\begin{equation}
2\sum\limits_{b\in O}\sum\limits_{i-i_0+b(j-j_0)=0} P\{(i,j)\} - 2\sum\limits_{b\in O}\sum\limits_{i-i_0+b(j-j_0)=n/2} P\{(i,j)\}=0\in \Z_n.
\label{eq16}
\end{equation}

Given $0\le i,j\le n-1$, the coefficient $P\{(i,j)\}$ appears in the left term of Equation \ref{eq16} if and only if $i-i_0+b(j-j_0) \equiv 0 \mod n$ for some $b \equiv 1 \mod 8$. This is equivalent to $d=\gcd (8(j-j_0),n) | \ i-i_0+j-j_0$. In this case, the number of times that $P\{(i,j)\}$ appears in the left term is $d/8$ (multiplied by the coefficient $2$). On the other hand, $P\{(i,j)\}$ appears in the right term of Equation \ref{eq16} if and only if $i-i_0+b(j-j_0) \equiv n/2 \mod n$ for some $b\in O$. This is equivalent to $d | \ i-i_0+j-j_0-n/2$, and the number of times $P\{(i,j)\}$ appears is also $d/8$. Therefore, if $n/8$ does not divide $j-j_0$, the coefficient $P\{(i,j)\}$ in Equation \ref{eq16} cancels. The unique coefficients that survive in the equation are those $P\{(i,j)\}$ in the left term with $j\equiv j_0 \mod n/8$ and $i\equiv i_0+j_0-j \mod n$, and those $P\{(i,j)\}$ in the right with $j\equiv j_0 \mod n/8$ and $i\equiv i_0+j_0-j +n/2 \mod n$. Each of them appears $d/8=n/8$ times, multiplied by the coefficient $2$. The coefficients in the left have positive sign and those in the right negative. But since $P\in H$, $P\{(i,j)\}$ is even for all $(i,j)$, so $\frac{n}{4}P\{(i,j)\}=-\frac{n}{4}P\{(i,j)\} \in \Z_n$. In conclusion $$\frac{n}{4} \sum\limits_{(i,j)\in \Gamma \cdot (i_0,j_0)} P\{(i,j)\}= 0\in \Z_n,$$ where $\Gamma \cdot (i_0,j_0)$ is the orbit of $(i_0,j_0)$ under $\Gamma=\Z_8 \times \Z_2$. In other words, $$ \sum\limits_{(i,j)\in \Gamma \cdot (i_0,j_0)} P\{(i,j)\}= 0\in \Z_4.$$ This means that if $P\in \ker (\psi)$, the congruence of $P\{(i_0,j_0)\}$ modulo $4$ is determined by the other coefficients $P\{(i,j)\}$ with $(i,j)$ in the orbit of $(i_0,j_0)$. Thus, the order of $\ker (\psi)$ is smaller than or equal to $\left( \left(\frac{n}{2} \right)^{15} \frac{n}{4} \right)^{\frac{n^2}{16}}$ (there are $\frac{n^2}{16}$ orbits and there are only $\frac{n}{4}$ choices for $P\{(i_0,j_0)\}$ once we have chosen the other $15$ values of $P\{(i,j)\}$ in the same orbit). Finally, $$|\textrm{Im}(\psi)|=\frac{|H|}{|\ker (\psi)|}\ge \frac{\left(\frac{n}{2} \right)^{n^2}}{\left(\frac{n}{2} \right)^{n^2} \frac{1}{2^{\frac{n^2}{16}}}}=2^{\frac{n^2}{16}}.$$
\end{proof}

Recall that for every $m\ge 1$, $\gamma_m(\F)/ \gamma_{m+1} (\F)$ is a free abelian group of finite rank and a basis is given by the basic commutators of weight $m$. The basic commutators of weight $1$ are $y$ and $x$. We set $y<x$. Once we have defined the basic commutators of weight smaller than $m$ and we have ordered them, we define those basic commutators of weight $m$: they are commutators $[c_j,c_i]$ where $c_j$ and $c_i$ are basic commutators of weight $j$ and $i$, $j+i=m$, $c_j<c_i$, and, if $i\ge 2$ and $c_i=[c_t,c_s]$ then $c_t\le c_j$. We set an arbitrary order among the basic commutators of weight $m$. Each of them is greater than any commutator of weight $<m$. Our basic commutators are the inverses of the classical ones, which are defined using the other convention for commutators $(a,b)=a^{-1}b^{-1}ab$ (see \cite[Chapter 11]{Hal2}). If $\M=\F/\F''$ is the free metabelian group of rank $2$, then most basic commutators are trivial in $\M$ and for $m\ge 2$, $\gamma_m(\M)/\gamma_{m+1}(\M)$ is generated by the basic commutators $e_{i,m-i-1}=[\underbrace{x,x,\ldots, x}_{i},\underbrace{y,y, \ldots, y}_{m-i-1},x]=[x,[x, \ldots, [x, [y,[y, \ldots, [y,x]]\ldots ]$, where $0\le i\le m-2$. For $i=0$, $e_{0,m-1}$ this is just the $(m-1)$-Engel word $e_{m-1}$.

\begin{prop} \label{basiccom}
Let $n=2^k$ for some $k\ge 3$. Then $e_{i,m-i-1}$ is non-trivial in the class $m$ factor $\gamma_{m}(M(2,n)) /\gamma_{m+1}(M(2,n))$ for every $2\le m\le n+2$ and every $0\le i\le m-2$.
\end{prop}  
\begin{proof}
It suffices to prove the statement for $m=n+2$. The winding invariant of $e_{i,n-i+1}$ is $-(X-1)^i(Y-1)^{n-i}$. Let $\inv :M(2,n)' \to \Z_n$ be the diagonal invariant associated to $\varphi :\Z \times \Z \to \Z_n$, where $\varphi(1,0)=\varphi (0,1)=1$, $t$ the identity of $\Z \times \Z$ and $c:\Z_n \to \{\textrm{black, white}\}$ the good coloring which paints $0,1,\ldots, \frac{n}{2}-1$ with black. Since $\inv (X^rY^s)=\inv (Y^{r+s})=\hh^0(Y^{r+s})$, then $\inv (e_{i,n-i+1})=\hh^0 (-(Y-1)^{n})=\hh^0 (e_{n+1})=\binom{n}{n/2}-2 \in \Z_n$ is not a multiple of $8$ by Example \ref{ocho} and Lemma \ref{accion}. On the other hand we will see that $\inv (z)$ is a multiple of $8$ for every $z\in \gamma_{n+3}(\F)$. The proof is similar to that of Lemma \ref{gama}. If $z\in \gamma_{n+3}(\F)$, then $P_z$ is a sum of multiples of polynomials of the form $P=(m_1-1)(m_2-1)\ldots (m_{n+1}-1)$ where the $m_i$ are monomials. Thus, it suffices to prove that for any good coloring $c':\Z_n \to \{\textrm{black, white}\}$ the diagonal invariant $\inv ': \F' \to \Z_n$ associated to $\varphi$, $t$ and $c'$, satisfies that $\inv '(P)$ is a multiple of $8$. Now, $\inv '(P(X,Y))=\inv '(P(Y,Y))$, and $P(Y,Y)$ is a multiple of $Q=(Y-1)^{n+1}$. Therefore, it suffices to prove that $\inv' (Q)$ is a multiple of $8$ for each good coloring $c'$.

Suppose first that $c'=c$, so $\inv '=\inv$. Then $$\inv (Q)=\sum\limits_{i=0}^{\frac{n}{2}-1} (-1)^{i+1} \binom{n+1}{i}-\sum\limits_{i=\frac{n}{2}}^{n-1} (-1)^{i+1} \binom{n+1}{i}+\sum\limits_{i=n}^{n+1} (-1)^{i+1} \binom{n+1}{i}.$$

Using that $\binom{n+1}{i}=\binom{n}{i-1}+\binom{n}{i}$ for $i\ge 1$, we obtain that $\inv (Q)=2\binom{n}{n/2-1}-2n$. Now, the action of $\Z_n$ on the set $\binom{\Z_n}{n/2-1}$ is free: if $S\subseteq \Z_n$ is a subset of cardinality $n/2-1$ and $\frac{n}{2} \cdot S=S$, then $|S|=2|S\cap [n/2]|$ is even, a contradiction. Thus, every orbit has cardinality $n$ and then $\binom{n}{n/2-1}$ is a multiple of $n$. This proves that $\inv (Q)=0\in \Z_n$, so in particular it is a multiple of $8$.

Suppose now that $c':\Z_n\to \{\textrm{black, white}\}$ is any good coloring. Then $c'$ can be obtained from $c$ by a sequence of changes of colors of pairs $j, j+\frac{n}{2}$ for $0\le j\le \frac{n}{2}-1$. Thus, it suffices to prove that if two good colorings $c', c''$ of $\Z_n$ differ only in $j$ and $j+\frac{n}{2}$, then $\inv ''(Q)-\inv '(Q)$ is a multiple of $8$ for the invariants $\inv ', \inv ''$ associated to $c',c''$. 

If $j=0$, then $\inv ''(Q)-\inv '(Q)=2(\binom{n+1}{0}-\binom{n+1}{n/2}+\binom{n+1}{n}) \equiv 2(2-\binom{n+1}{n/2}) \mod 8$. Again, $\binom{n+1}{n/2}=\binom{n}{n/2-1}+\binom{n}{n/2}$ and the first term is a multiple of $n$. By the argument in Lemma \ref{accion}, $\binom{n}{n/2} \equiv 2 \mod 4$, so $\inv ''(Q)-\inv '(Q)$ is a multiple of $8$.

If $j=1$, $\inv ''(Q)-\inv '(Q)=2(-\binom{n+1}{1}+\binom{n+1}{n/2+1}-\binom{n+1}{n+1})\equiv 2(\binom{n+1}{n/2}-2) \equiv 0 \mod 8$ by the previous paragraph.

Suppose now that $2\le j \le \frac{n}{2}-1$. Then $\inv ''(Q)-\inv '(Q)= (-1)^j2(\binom{n+1}{j}-\binom{n+1}{n/2+j})=(-1)^j2(\binom{n}{j-1}+\binom{n}{j}-\binom{n}{n/2+j-1}-\binom{n}{n/2+j})$. If $0\le r \le n$ is different from $0, \frac{n}{2}, n$, then the action of $\Z_n$ on $\binom{\Z_n}{r}$ has no orbits of cardinality $1$ or $2$. In this case, $4|\binom{n}{r}$. Since $j-1,j,\frac{n}{2}+j-1$ and $\frac{n}{2}+j$ are different from $0, \frac{n}{2}, n$, then $\inv ''(Q)-\inv '(Q)$ is a multiple of $8$.
\end{proof}

Proposition \ref{basiccom} shows that the order of each of the $m-1$ generators $e_{i,m-i-1}$ of the class $m$ factor $\gamma_{m}(M(2,n)) /\gamma_{m+1}(M(2,n))$ is at least $2$. If moreover one could prove that the order of this factor is greater than or equal to $2^{m-1}$, for $2\le m\le n+2$, then $|M(2,n)'|\ge 2^{\frac{(n+1)(n+2)}{2}}$, and this would provide a better bound than that given in Theorem \ref{cotainf}.

\section{Upper bounds for $|M(2,2^k)|$}

As explained in Section \ref{recall}, Newman proved \cite{New} that $|M(2,2^k)|\le 2^{3k+k2^{2k}}$. In general this bound is far from optimal. In this section we use the winding invariant map to find better upper bounds. We will focus on the special case $k=2$, were the bounds found are sharp.

We recall from \cite{wind} the construction of the winding invariant map associated to a group with a \textit{cocommutative} presentation. A Laurent polynomial $P\in \ZZ$ can be represented graphically by writing each non-zero coefficient $P\{(i,j)\}$ in the square $(i,j)$ of the grid. The graphical representation of a polynomial is called a \textit{piece}, so we will also say that a polynomial \textit{is} a piece. The piece $P_w=W(w)$ associated to an element $w\in \F'$ has inside each square, the winding number of the curve $\gamma_w$ around the center of that square.  

In \cite[Proposition 18]{wind} it is proved that for a group $G$ presented by $\langle x,y| S\rangle$ where $S$ is a subset of $\F'$, $W:\F' \to \ZZ$ induces an epimorphism $\overline{W}:G' \to \ZZ/I$ where $I$ is the ideal of $\ZZ$ generated by the polynomials $P_s$ with $s\in S$. Moreover $\textrm{ker}(\overline{W})=G''$, so there is an induced isomorphism $G'/G''\to \ZZ/I$, which we also denote by $\overline{W}$. We apply this to the group $G$ presented by $\langle x,y| \ \F' \cap \F^n\rangle$. We have then that $G'/G''$ is isomorphic to $M(2,n)'$, so there is an isomorphism $\overline{W}:M(2,n)'\to \ZZ/I$, where $I$ is the ideal generated by (consisting of) the polynomials $P_s$ with $s\in S= \F' \cap \F^n$. The obvious observation is then:

\begin{obs} \label{obspiezas}
The order of $M(2,n)'$ is smaller than or equal to the cardinality of any set $A\subseteq \ZZ$ of representatives of $\ZZ/I$.
\end{obs}

We will understand the quotient $\ZZ/I$ graphically. The sum of two pieces $P,Q\in \ZZ$ is the piece which is obtained by adding the two numbers of each square. The product of a piece $P$ by a monomial $X^iY^j$ is the translation of the piece by the vector $(i,j)\in \Z \times \Z$. All the computations that follow can be done algebraically, but we believe the graphical representation makes proofs simpler.

\begin{ej} \label{ej210}
Gupta and Tobin \cite{GT} proved that $|M(2,4)|=2^{10}$. Since $M(2,4)/M(2,4)'=\Z_2\times \Z_2$, this is equivalent to proving that $|M(2,4)'|=2^6$. In Example \ref{ejm24} we gave an alternative proof that $|M(2,4)'|\ge 2^6$. We will use the representation by pieces together with Remark \ref{obspiezas} to give an alternative proof that $|M(2,4)'|\le 2^6$, thus providing a different proof of Gupta and Tobin's result. In turn, this will give a new algorithm to decide whether an element in $M(2,4)$ is trivial (different from the one given in Example \ref{ejm24}). 

Step 1. Suppose first that $k$ is an arbitrary positive integer and $n=2^k$. Let $[x,y]=xyx^{-1}y^{-1}\in \F$. Then $w=[x,y]^n\in S= \F' \cap \F^n$ and $W(w)=n\in \ZZ$. Therefore, the piece that has the number $n$ in the square $(0,0)$ and all the other numbers equal to zero, is in $I$. Hence, all the pieces whose coefficients are divisible by $n$ are also in $I$. This proves that any element in $\ZZ/I$ has a representative with all the coefficients in the interval $\{0,1,2,\ldots, n-1\}$.

Step 2. Let $u=([x,y]y[y,x])^ny^{-n}\in S$. Then $W(u)=1-Y^n\in I$. Thus, the piece in Figure \ref{pieza1} lies in $I$. For any piece in $\ZZ$ we can add and subtract translates of $u$ to obtain a representative whose only non-zero coefficients lie in the rows $0,1,\ldots, n-1$. A symmetric argument shows that in fact there is a representative with all the non-zero coefficients in the square $[0,n]\times [0,n]$ of rows $0$ to $n-1$ and columns $0$ to $n-1$. We say that the representative is \textit{concentrated} in the square $[0,n]\times [0,n]$.

\begin{figure}[h] 
\begin{center}
\includegraphics[scale=0.7]{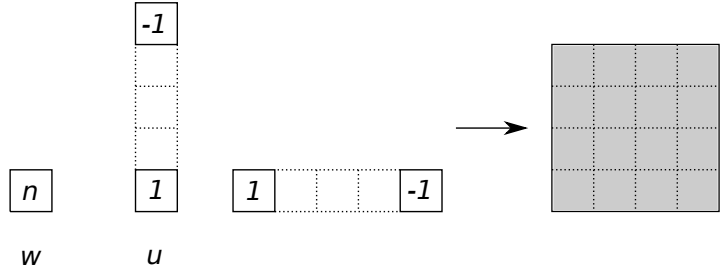}
\caption{The pieces associated to $w=[x,y]^n$, $u=([x,y]y[y,x])^ny^{-n}$ and $([x,y]x[y,x])^nx^{-n}$ turn any element of $\ZZ$ to a representative concentrated in $[0,n]\times [0,n]$ with all the coefficients in the interval $[0,n-1]$. In this picture $n=4$.}\label{pieza1}
\end{center}
\end{figure}

This argument shows that $|M(2,n)'|$ is smaller than or equal to $n^{n^2}=2^{k2^{2k}}$, so $|M(2,n)|\le n^{n^2+2}$. 

Step 3. Let $v=(xyx^{-1})^ny^{-n}\in S$, so $W(v)=1+Y+Y^2+\ldots + Y^{n-1}\in I$ is the piece which appears in Figure \ref{pieza2}. Adding multiples of horizontal translates of this piece we can turn any piece concentrated in the square $[0,n]\times [0,n]$ to another piece concentrated in the rectangle $[0,n]\times [0,n-1]$. And by a symmetric argument, for any element in $\ZZ$ there is a representative concentrated in the square $[0,n-1]\times [0,n-1]$ with all the coefficients between $0$ and $n-1$. Therefore $|M(2,n)'|\le n^{(n-1)^2}$ (so $|M(2,n)|\le n^{(n-1)^2+2}$).

\begin{figure}[h] 
\begin{center}
\includegraphics[scale=0.7]{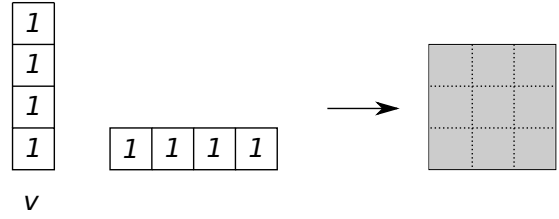}
\caption{The pieces associated to $v=(xyx^{-1})^ny^{-n}$ and $x^n(yx^{-1}y^{-1})^n$ turn any element of $\ZZ$ concentrated in $[0,n]\times [0,n]$ to a representative concentrated in $[0,n-1]\times [0,n-1]$. Together with piece $w$, we can assume each coefficient is in the interval $[0,n-1]$.}\label{pieza2}
\end{center}
\end{figure}

Before going further, let us summarize some simple yet useful remarks. If $Q\in \ZZ$ is a piece in $I$, then any translate of $Q$ is in $I$ and any multiple as well. In particular $-Q\in I$, which is the piece that contains in each square of the grid the opposite of the coefficient $Q$ has. Also, if $s=s(x,y)\in S=\F'\cap \F^n$, then $s(x^{-1},y)$ also lies in $S$. Thus if $Q=Q(X,Y)$ is a piece in $I$, the symmetric piece $Q(X^{-1},Y)$ with respect to the $y$-axis also lies in $I$. A similar argument shows that rotations of $Q$ by $\pi/2, \pi$ and $3\pi/2$ are also in $I$. Therefore $I$ is invariant by translations, rotations of angle multiple of $\pi/2$ and symmetries about the axes and the lines $y=x$ and $y=-x$.

Remark. Suppose $A$ is a region (a union of squares of the grid), and that a piece $P\in \ZZ$ is concentrated in $A$. That is, all the coefficients of $P$ out of $A$ are zero. If $Q$ is a piece in $I$ which is also concentrated in $A$, and if $(i,j)$ is a square of $Q$ with coefficient $1$ or $-1$, then there is a representative of $P$ in $\ZZ/I$ concentrated in the region $B$, obtained from $A$ by removing the square $(i,j)$. This representative is just $P$ plus a multiple of $Q$. This is the idea we already used in step 3 (and step 2 as well).

Step 4. The first three steps worked for arbitrary $k$. From now on we will assume $k=2$, $n=4$. Let $t=x^4(x^{-1}y)^4y^{-4}\in S$. The piece corresponding to $t$ appears in Figure \ref{pieza3}, together with a translate $T$ of a symmetric. The difference between the two pieces is $Q_1\in I$.  Subtracting a translate of $Q_1$ to itself we obtain the piece $Q_2\in I$. 

\begin{figure}[h] 
\begin{center}
\includegraphics[scale=0.7]{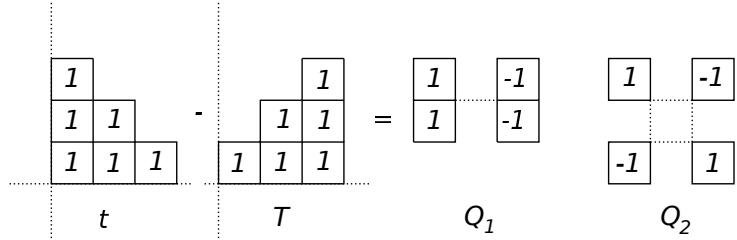}
\caption{The piece associated to $t=x^4(x^{-1}y)^4y^{-4}$ and a symmetric $T$. At the right, the pieces $Q_1$ and $Q_2$. }\label{pieza3}
\end{center}
\end{figure}

We apply the idea of the Remark to the piece $Q_1\in I$. This can be used to transform an element $P\in \ZZ$ concentrated in the square $[0,3]\times [0,3]$ into a representative concentrated in a region with one square less (see Figure \ref{pieza4}). We do this two more times with $Q_1$ and a rotation of $Q_1$ to find a representative concentrated in the stair of three steps.

\begin{figure}[h] 
\begin{center}
\includegraphics[scale=0.7]{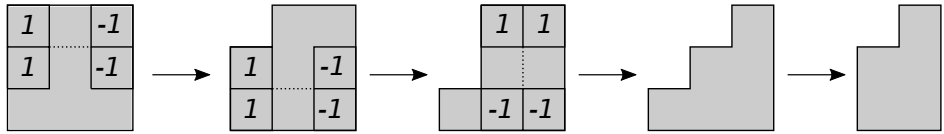}
\caption{An element concentrated in $[0,3]\times [0,3]$ has a representative concentrated in the region of five squares.}\label{pieza4}
\end{center}
\end{figure}

We then use the piece $T$ in Figure \ref{pieza3} to obtain a representative concentrated in the last region of Figure \ref{pieza4}, which consists of just five squares.

Step 5. Let $t'=x^8(x^{-2}y)^4y^{-4}\in S$. The piece associated to $t'$ appears in Figure \ref{pieza6}. If we subtract the piece associated to $t$, then we subtract (a translate of) the horizontal piece in Figure \ref{pieza2}, then piece $Q_2$, and, finally piece $T$, we obtain a piece with only two non-zero coefficients, both equal to $-2$. This is piece $D\in I$.

\begin{figure}[h] 
\begin{center}
\includegraphics[scale=0.7]{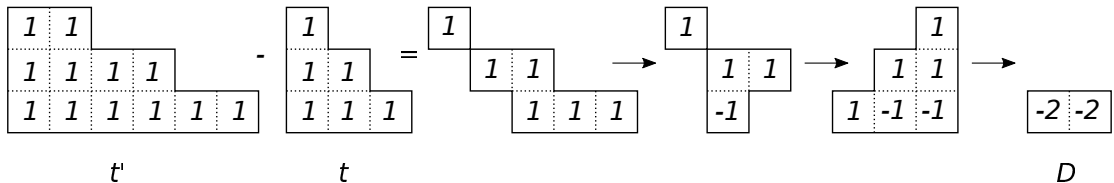}
\caption{Using pieces from $I$ we obtain piece $D\in I$.}\label{pieza6}
\end{center}
\end{figure}

Step 6. From previous steps we know that any element $P\in \ZZ$ has a representative in $\ZZ/I$ concentrated in the region of five squares in Figure \ref{pieza4}. Adding multiples of piece $D$ and of its vertical version, we can find a representative of $P$ concentrated in the same five-square region, and where four of the five coefficients inside are $0$ or $1$. Finally we can use the piece associated with $w$ (step 1), to make the last coefficient lie in the interval $\{0,1,2,3\}$. There are only $2^6$ pieces satisfying these conditions, so the order of $\ZZ/I$ is at most this number.

\begin{figure}[h] 
\begin{center}
\includegraphics[scale=0.7]{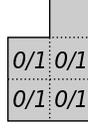}
\caption{The representative of $P$ is one among $64$ pieces.}\label{pieza7}
\end{center}
\end{figure}

This finishes the proof that $|M(2,4)|\le 2^{10}$.

\end{ej} 

In the previous example we have proved the following bound for arbitrary $k\ge 1$ and $n=2^k$: $|M(2,n)|\le n^{(n-1)^2+2}$. In fact, using the piece corresponding to $t=x^n(x^{-1}y)^ny^{-n}$ twice, we can obtain a representative concentrated in the region $[0,n-1]\times [0,n-1]$ with two squares removed. Note that this part of the proof does not use the fact that $n$ is a power of $2$. It works for an arbitrary integer $n\ge 3$. Thus we have proved the following

\begin{prop} \label{cotasup}
Let $n\ge 3$ be an integer. Then $|M(2,n)|\le n^{(n-1)^2}$.
\end{prop}

\section{A lower bound for the order of the restricted Burnside group $R(2,8)$}

Although Zel'manov \cite{Zel1,Zel2} proved that the restricted Burnside group $R(d,m)$ exist for all $d,m$, its order is known in very few cases (see Section \ref{recall} above). As far as we know, the best lower bound for $|R(2,8)|$ available in the literature is given in \cite{GHMN}: $2^{4109}$. In fact, this is the order of the group $\F/(\F^4)^2$. Indeed, $|\F/\F^4|=|B(2,4)|=2^{12}$, so by Schreier's formula, $\rk (\F^4)=1+2^{12}=4097$. Now $(\F^4)' \leqslant (\F^4)^2$, so $\F^4/(\F^4)^2$ is abelian of order $2^{\rk \F^4}$. Thus, $|\F/(\F^4)^2|=|\F/\F^4||\F^4/(\F^4)^2|=2^{4109}$. We will use our invariants to show that this bound is not sharp. Moreover we will prove the following

\begin{prop} \label{restric}
$|R(2,8)|\ge 2^{4115}$.
\end{prop}

It is not hard to see that $\F^8\F'''$ is a finite-index subgroup of $\F$. Let $R=\F^8\F'\leqslant \F$, $S=R^8R'\leqslant \F$ and $T=S^8S'\leqslant \F$. Then $|\F:R|$ is finite, so $R$ has finite rank. Similarly $|R:S|$ is finite and $|S:T|$ is finite. Then $\F/T$ is finite. In fact we can compute its order explicitly, but we will not need that. Also, it is easy to prove that $T\leqslant \F^8\F'''$. It suffices to show that $S \leqslant \F^8\F''$, and similarly this follows from $R\leqslant \F^8\F'$. Thus, $\F^8\F'''$ is a finite-index subgroup of $\F$.

Since both $\F^8\F'''$ and $(\F^4)^2$ are finite-index subgroups of $\F$, $\frac{\F}{(\F^8\F''')\cap (\F^4)^2}$ is a finite group of exponent $8$, and its order gives a lower bound for $|R(2,8)|$. $$\left|\frac{\F}{(\F^8\F''')\cap (\F^4)^2}\right|=\left|\frac{\F}{(\F^4)^2} \right| \left|\frac{(\F^4)^2}{(\F^8\F''')\cap (\F^4)^2} \right|.$$

Let $\phi: \frac{(\F^4)^2\cap \F'}{(\F^8\F''')\cap (\F^4)^2 \cap \F'} \to (\Z_8)^8$ be the map induced by the restriction of $\Omega : \F' \to (\Z_8)^8$. Note that $\phi$ is well-defined since $\F^8\F''' \cap \F' \leqslant \F^8\F'' \cap \F'$. Let $z=(x^4(x^{-1}y)^4y^{-4})^2 \in (\F^4)^2\cap \F'$. It is easy to see that $\phi (z)=(4,0,0,4,4,0,0,4)$. Moreover, considering the conjugates $x^iy^jzy^{-j}x^{-i}$ with $i,j=0,1,2$, we obtain the nine $8$-tuples $(a_0,a_1,a_2,a_3,b_0,b_1,b_2,b_3)$ where $(a_0,a_1,a_2,a_3), (b_0,b_1,b_2,b_3)\in \{(4,0,0,4), (4,4,0,0), (0,4,4,0)\}$. These $8$-tuples generate a subgroup of $(\Z_8)^8$ of order $2^6$. Thus $|\frac{(\F^4)^2\cap \F'}{(\F^8\F''')\cap (\F^4)^2 \cap \F'}|\ge 2^6$. In particular $|R(2,8)|\ge 2^{4109}2^6=2^{4115}$.

\section{Other primes, other colors}

Up to this point we have only studied the case that the exponent $n$ is a power of $2$. Many of the techniques we developed, naturally extend to powers of different primes. In this section we will generalize the notion of good coloring. In fact this is a generalization of a variant of the notion we know already. For the case $p=2$ this new notion would have led to the same results already proved.

Let $p$ be a positive prime number. Given a subset $S$ of the integers and $d$ a positive integer, a \textit{$p$-matching of $S$ of difference $d$} is a partition of $S$ in which every part is an arithmetic progression of $p$ terms and common difference $d$.
We have the following version of Lemma \ref{facil}:

\begin{lema}
Let $n$, $d$ be positive integers. If $pd$ divides $n$, then there is a $p$-matching of $[n]=\{0,1,\ldots, n-1\}$ of difference $d$.
\end{lema}

Let $k$ be a positive integer and let $n=p^k$. A \textit{$p$-good coloring of $\Z_n$} is a function $c:\Z_n\to \Z_n$ such that for every arithmetic progression $a,a+\frac{n}{p},a+2\frac{n}{p},\ldots ,a+(p-1)\frac{n}{p}$ of $p$ terms and difference $\frac{n}{p}$, we have that the sum $c(a)+c(a+\frac{n}{p})+c(a+2\frac{n}{p})+c(a+(p-1)\frac{n}{p})$ of its colors is zero in $\Z_n$. Note that a $2$-good coloring is not exactly the same as a good coloring in the sense of previous sections. We have the following lemma.

\begin{lema} \label{lemap}
Let $k\ge 1$, $n=p^k$. Let $c:\Z_n \to \Z_n$ be a $p$-good coloring. 

(i) Let $a, a+b, a+2b, \ldots, a+(n-1)b$ be an arithmetic progression of length $n$ in $\Z$. Then we have $c(a)+c(a+b)+c(a+2b)+ \ldots +c(a+(n-1)b)=0\in \Z_n$.

(ii) Let $a, a+b, a+2b, \ldots, a+(\frac{n}{p}-1)b$ be an arithmetic progression of length $\frac{n}{p}$ in $\Z$ with $b$ a multiple of $p$. Then the sum of the colors of its terms is a multiple of $\frac{n}{p}$.
\end{lema} 

The proofs of these lemmas are similar to the proofs of Lemmas \ref{facil} and \ref{facil2}, so we omit them.

Given a function $\varphi: \Z \times \Z \to \Z_n$ and a $p$-good coloring $c$ of $\Z_n$, define the associated invariant $\inv=\inv_{\varphi, c} :\ZZ \to \Z_n$ by $$\inv (P)= \sum\limits_{(i,j)\in \Z \times \Z} c\varphi (i,j) P\{(i,j)\} \in \Z_n.$$ For $z\in \F'$ we define $\inv (z)=\inv (P_z)$.

\begin{teo} \label{muchosinvariantesp}
Let $p$ be a positive prime, $k\ge 2$, $n=p^k$. Let $c$ be a $p$-good coloring of $\Z_n$ and let $\varphi:\Z \times \Z\to \Z_n$ be a homomorphism such that $\varphi(1,0)$, $\varphi(0,1) \in \Z_n$ are invertible. Let $t:\Z \times \Z \to \Z \times \Z$ be a translation by any vector $(i_0,j_0)\in \Z \times \Z$. Let $\inv=\inv_{\varphi t, c}$ be the associated invariant. Then for any $z\in \F'$ which is a product of $n$th powers in $\F$, $\inv(z) \in \Z_{n}$ is trivial in $\Z_{\frac{n}{p}}$.
\end{teo}
\begin{proof}
The proof is almost identical to the proof of Theorem \ref{muchosinvariantes}. The difference is perhaps in case 4. The cost that we pay for using $p\neq 2$ and the general notion of $p$-good coloring is that $\inv (z)$ may be nontrivial in $\Z_n$, so we have to descend to $\Z_{\frac{n}{p}}$.

We may assume $(i_0,j_0)=(0,0)$ and that $z$ is a product of three $n$th powers. By Lemma \ref{lemap} we only have to study the invariant for stair-like regions.

Case 1: $p \nmid h$, $p| l$. Inside the stair of $n-1$ steps, consider the region $A_0$ formed by the top $\frac{n}{p}-1$ levels (see Figure \ref{caso1p}). We define the regions $A_1, A_2, \ldots, A_{p-1}$, where $A_{i+1}$ is the translation of $A_i$ by $\frac{hn}{p}$ units downwards.

\begin{figure}[h] 
\begin{center}
\includegraphics[scale=0.9]{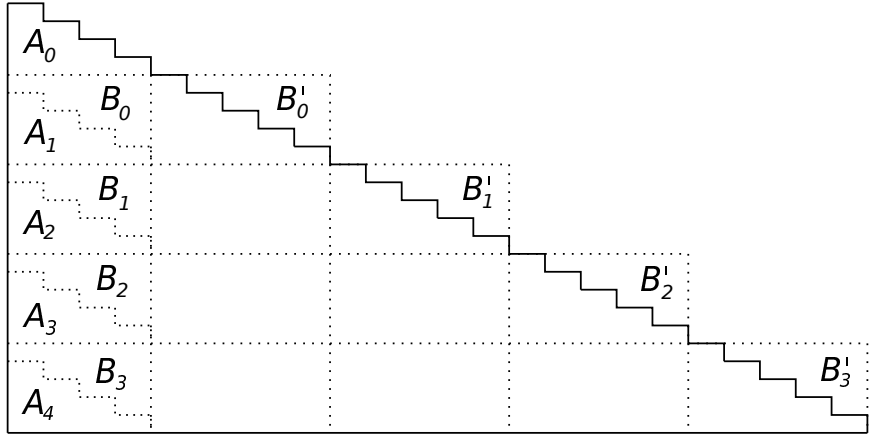}
\caption{Here $p=5$, $n=5^2$, $p \nmid h$ and $p | l$.}\label{caso1p}
\end{center}
\end{figure}

Since $p \nmid \varphi (0,1)$ and $c$ is $p$-good, the sum of the invariants of the regions $A_i$ is trivial in $\Z_n$. For $0\le i \le p-2$ we define $B_i$ as the region between $A_i$ and $A_{i+1}$. $B_i'$ is the translation of $B_i$ by $(i+1)\frac{ln}{p}$ units to the right, so it is outside the original stairs. Since $n| \frac{(i+1)ln}{p}$, $B_i'$ and $B_i$ are painted in the same way. Therefore the invariant $\inv$ of the original stairs is a sum of invariants of rectangles with one side of length divisible by $n$, and thus, trivial in $\Z_n$.

Case 2: $p | h$, $p \nmid l$ is symmetric to case 1.

Case 3: $p| h$, $p | l$. We must show that the stairs with $n-1$ steps have trivial invariant. We will change here the proof we did in Theorem \ref{muchosinvariantes} just a bit to make the argument simpler. We can add one more step to the stairs without changing the value of $\inv$, since the $l$ new columns that we added have height $hn$, a multiple of $n$.   

The $n$-step stair appears in Figure \ref{caso3p}. We consider region $A_0$ inside the top $n/p$ levels of the stairs. $A_0$ is an $n$-step stair itself but with height $h/p$ and length $l/p$. We define $A_1,A_2, \ldots, A_{p-1}$ where $A_{i+1}$ is obtained from $A_i$ by a translation of $\frac{n}{p}(l,-h)$. By induction each stair $A_i$ has invariant $\inv$ trivial in $\Z_{\frac{n}{p}}$. In fact this holds when we remove one step from $A_i$, but again, removing columns of height $\frac{hn}{p}$ does not change $\inv$ since this is a multiple of $n$. The region below the $A_i$'s is a union of rectangles of size $\frac{ln}{p} \times \frac{hn}{p}$, so its invariant is trivial in $\Z_n$. The region above the $A_i$'s appears in Figure \ref{caso3p} with the sign $+$. This region is a union of $n$ regions of the form $C$, $C+v$, $C+2v$, $\ldots$, $C+(n-1)v$ for certain vector $v$ and certain region $C$. By Lemma \ref{lemap}, its invariant is trivial in $\Z_n$.

\begin{figure}[h] 
\begin{center}
\includegraphics[scale=0.7]{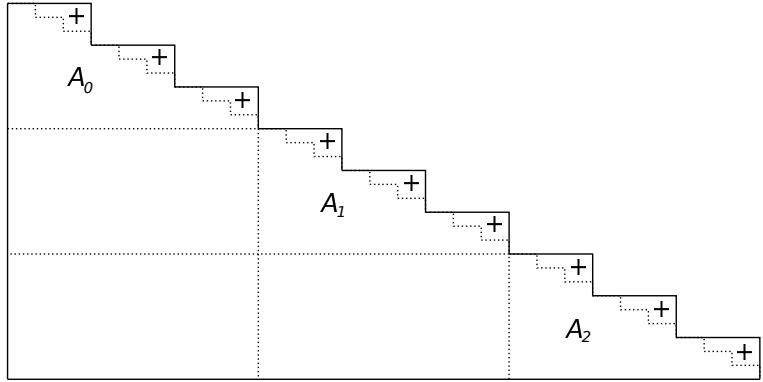}
\caption{Here $p=3$, $n=3^2$, $p | h$ and $p | l$.}\label{caso3p}
\end{center}
\end{figure}

Case 4: $p \nmid h$, $p \nmid l$. We have arrived to the crux of the proof, where we have to pay the price that $\inv$ is trivial in $\Z_{\frac{n}{p}}$ and maybe not in $\Z_n$. Consider the regions $A_0, A_1, \ldots , A_{p-1}$ defined as in case 1 (see Figure \ref{caso4p}).

\begin{figure}[h] 
\begin{center}
\includegraphics[scale=1]{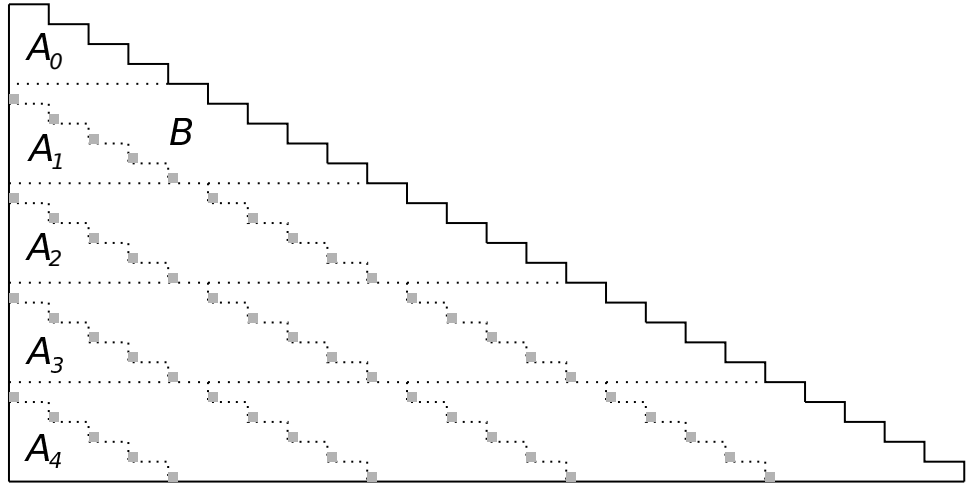}
\caption{Here $p=5$, $n=5^2$, $p \nmid h$ and $p \nmid l$.}\label{caso4p}
\end{center}
\end{figure}	

Let $B$ be the region inside the stairs which is at the right of $A_1$. The complement $C$ of the region formed by the $A_i$'s is a union of $\frac{p(p-1)}{2}$ translates of $B$. Take the square $Q$ in $B$ which is a neighbor of $A_1$ and is in the left margin of the stairs. Let $v=(l,-h)$, $v'=(0,-h)$. Consider the $\frac{(p-1)n}{2}$ translates of $Q$ which appear with gray in Figure \ref{caso4p}. They are

(line $0$)  $Q,Q+v,Q+2v, \ldots , S+(\frac{(p-1)n}{p}-1)v,$

(line $1$) $Q+\frac{n}{p}v', Q+\frac{n}{p}v'+v, Q+\frac{n}{p}v'+2v, \ldots, Q+\frac{n}{p}v'+(\frac{(p-2)n}{p}-1)v,$

(line $2$) $Q+\frac{2n}{p}v', Q+\frac{2n}{p}v'+v, Q+\frac{2n}{p}v'+2v, \ldots, Q+\frac{2n}{p}v'+(\frac{(p-3)n}{p}-1)v,$

$\ldots$

(line $p-2$) $Q+\frac{(p-2)n}{p}v', Q+\frac{(p-2)n}{p}v'+v, Q+\frac{(p-2)n}{p}v'+2v, \ldots, Q+\frac{(p-2)n}{p}v'+(\frac{n}{p}-1)v$. 

If $b=\varphi (v) \in \Z_n$ is not invertible, then by Lemma \ref{lemap} (ii), the sum of the invariants of the translates in each of this lines is trivial in $\Z_{\frac{n}{p}}$, so the invariant of the grey region is trivial in $\Z_{\frac{n}{p}}$. 

If $b=\varphi (v)$ is invertible in $\Z_n$, then we claim that we can divide the $p-1$ lines above in groups, in such a way that for each group, the sequences obtained after applying $\varphi$ can be concatenated to form an arithmetic progression in $\Z_n$ of length divisible by $n$. 

Let $a=\varphi (Q)$, $b'=\varphi (v')$. When we apply $\varphi$ to line 0 above, we obtain an arithmetic progression of difference $b$ which begins in $a$ and finishes in $a+(\frac{(p-1)n}{p}-1)b$. So, the next term of this sequence should be $a+\frac{(p-1)n}{p}b$.

After applying $\varphi$, line 1 begins with $a+\frac{n}{p}b'$ and finishes with $a+\frac{n}{p}b'+(\frac{(p-2)n}{p}-1)b$, so the next term should be $a+\frac{n}{p}b'+\frac{(p-2)n}{p}b$.

We do this for each line. Line $i$ begins with $a+i\frac{n}{p}b'$ and the next term after the last one should be $a+i\frac{n}{p}b'+\frac{(p-(i+1))n}{p}b$. In order to prove that these lines can be glued to form longer arithmetic progressions, we will prove the following Claim: the subsets $\{a+i\frac{n}{p}b'\}_{0\le i\le p-2}$ and $\{a+i\frac{n}{p}b'+\frac{(p-(i+1))n}{p}b\}_{0\le i\le p-2}$ of $\Z_n$ are equal and each of them has $p-1$ elements. 

This is equivalent to proving that the subsets $\{ib'\}_{0\le i\le p-2}$ and $\{ib'-(i+1)b\}_{0\le i\le p-2}$ of $\Z_p$ are equal and that each has $p-1$ elements. That their cardinality is $p-1$ follows from the fact that $b'$ and $b'-b$ are invertible in $\Z_p$ and this, in turn, follows from the hypothesis that $\varphi(1,0)$ and $\varphi (0,1)$ are invertible and that $p\nmid hl$. In order to see that both sets are equal, it suffices to observe that the missing element $(p-1)b'$ of the first set is equal to the missing element $(p-1)b'-pb$ in the second. 

The Claim says that we can glue the $p-1$ arithmetic progressions to form arithmetic cycles in $\Z_n$ of difference $b$. By arithmetic cycle of difference $b$ we mean an arithmetic progression $e,e+b, e+2b, \ldots, e+(t-1)b \in \Z_n$ such that $e+tb=e\in \Z_n$. Since $b\in \Z_n$ is invertible, the length $t$ of one such cycle must be a multiple of $n$. Thus, the invariant $\inv$ of the grey region is the sum of the colors of arithmetic progressions of lengths divisible by $n$. By Lemma \ref{lemap} (i), the invariant of the grey region is trivial in $\Z_{\frac{n}{p}}$. 

Now, region $C$ is a union of translates of the grey region, so we have that $\inv (C)=0\in \Z_{\frac{n}{p}}$ and then the invariant of the stairs is also trivial in $\Z_{\frac{n}{p}}$.
\end{proof}

In the hypotheses of Theorem \ref{muchosinvariantesp}, $\inv (z)$ need not be trivial in $\Z_n$. This fails even for $p=2$. Let $n=4$. Consider the $2$-good coloring $c:\Z_4 \to \Z_4$ defined by $c(0)=c(2)=0$, $c(1)=1$, $c(3)=3$. Let $\varphi : \Z \times \Z \to \Z_n$ be defined by $\varphi (1,0)=\varphi (0,1)=1$. Then $\inv_{\varphi, c}(x^4(x^{-1}y)^4y^{-4})=2\neq 0\in \Z_4$.

\section{The free nilpotent group of class $2$ and exponent $n$} \label{final}

In the previous sections we tried to understand some features of the free metabelian group $M(2,n)$ of rank $2$ and exponent $n$. The free nilpotent group of class $2$, rank $2$ and exponent $n$ is the group $N(2,n)=\frac{\F}{\gamma_3(\F) \F^n}$, where $\gamma_3(\F)=[\F,\F']$ denotes the third term in the lower central series of $\F$. Since $\F'' \leqslant \gamma_3(\F)$, $N(2,n)$ is a quotient of $M(2,n)$. This group is much easier to understand, even for arbitrary $n\ge 1$, using invariants. In fact the area invariant $A:\F' \to \Z$, $A(z)=P_z(1,1)$ describes $N(2,n)$ completely in the following sense.

We already know by Remark \ref{remark} that if $z \in \gamma_3(\F)$, then $P_z$ is a sum of multiples of polynomials of the form $m-1$ where $m$ is a monomial, so the area $A(z)=0$. Conversely, each $P\in \ZZ$ with $P(1,1)=0$ is a linear combination of $(X-1)$ and $(Y-1)$ so, $A(z)=0$ implies that $z\in \gamma_3(\F)$. If $n$ is even, denote by $\overline{A}:\F' \to \Z_{\frac{n}{2}}$ the composition of $A$ with the projection. If $n$ is odd, take $\overline{A}: \F' \to \Z_n$, instead. Theorem 59 in \cite{wind} implies that $\gamma_3(\F)\F^n \cap \F' \leqslant \ker (\overline{A})$. Conversely, if $z \in \ker (\overline{A})$, there exists $w\in \F^n$ with the same area: indeed $A(x^n(x^{-1}y)^ny^{-n})=\frac{n(n-1)}{2}$ and $A([x,y]^n)=n$, so a product of powers of those two elements has area $\gcd (\frac{n(n-1)}{2}, n)$, which is $\frac{n}{2}$ or $n$, depending on the parity of $n$. Thus $A(w^{-1}z)=0$, which shows that $w^{-1}z\in \gamma_3(\F)$, so $z\in \gamma_3(\F)\F^n \cap \F'$. Since $\overline{A}$ is surjective, we deduce that $N(2,n)'$ is isomorphic to $\Z_{\frac{n}{2}}$ for even $n$ and to $\Z_n$ for odd $n$. Since the abelianization of $N(2,n)$ has order $n^2$, we conclude that $|N(2,n)|=\frac{n^3}{2}$ for even $n$  and $|N(2,n)|=n^3$ for odd $n$.     

So, this is a case where the idea of invariants works pretty well. This case is of course much simpler than the metabelian.

In \cite{San}, Sanov proved that for any prime $p$ and $1\le i \le k$, the Engel congruence $e_{ip^i-1} (x,y)^{p^{k-i}} \in \gamma_{ip^i+1}(\F)\F^{p^k}$ holds. When $p=2$ and $i=1$, this reduces to $[x,y]^{2^{k-1}} \in  \gamma_3(\F) \F^{2^k}$. An alternative proof of this particular case has been given by Struik in \cite[Theorem 1]{Str} in which she proves that $[x,y]^{\frac{n}{2}}\in \gamma_3(\F)\F^{n}$ for each even positive integer $n$. Note that this follows immediately from the exposition above: $A([x,y]^{\frac{n}{2}})=\frac{n}{2}$, so $[x,y]^{\frac{n}{2}} \in \ker (\overline{A})=\gamma_3(\F)\F^n \cap \F'$. In fact, the area invariant gives a complete description of all the elements in $\gamma_3(\F)\F^n$.


 %
%
%
%
%
%


\end{document}